\documentclass[12pt]{amsart}
\usepackage[margin=1in]{geometry}
\usepackage{amssymb,amsfonts,amsmath}
\usepackage{enumerate}
\usepackage{mathrsfs}
\usepackage{color}
\usepackage{hyperref}
\usepackage[capitalise]{cleveref}
\usepackage{constants}
 
\newtheorem{thm}{Theorem}[section]
\newtheorem{cor}[thm]{Corollary}

\newtheorem{prop}[thm]{Proposition}
\newtheorem{lem}[thm]{Lemma}

\newtheorem{quest*}{Question}
\newtheorem{prob*}{Problem}

\theoremstyle{definition}

\theoremstyle{remark}
\newtheorem{remark}[thm]{Remark}

\theoremstyle{remark}

\numberwithin{equation}{section}
\crefname{figure}{Figure}{Figures}
\theoremstyle{plain}
\newtheorem*{thm*}{Theorem}
\crefname{thm}{Theorem}{Theorems}
\crefname{cor}{Corollary}{Corollaries}
\newtheorem*{cor*}{Corollary}
\crefname{cor*}{Corollary}{Corollaries}
\crefname{lem}{Lemma}{Lemmas}
\crefname{prop}{Proposition}{Propositions}
\crefname{conj}{Conjecture}{Conjectures}
\newtheorem*{conj*}{Conjecture}
\crefname{conj*}{Conjecture}{Conjectures}
\crefname{defn}{Definition}{Definitions}
\crefname{hyp}{Hypothesis}{Hypotheses}

\newcommand{\Z}{\mathbb{Z}}

\newcommand{\R}{\mathbb{R}}
\newcommand{\Q}{\mathbb{Q}}

\newcommand{\dsum}{\mathop{\sum\vspace*{-3pt}\sum}}

\renewcommand{\bar}{\overline}

\renewcommand{\epsilon}{\varepsilon}

\newcommand{\kf}{\mathfrak{f}}

\newcommand{\Li}{\mathrm{Li}}

\renewcommand{\Im}{\mathrm{Im}}

\renewcommand{\pmod}[1]{\, (\mathrm{mod} {\, #1})}

\newcommand{\kn}{\mathfrak{n}}
\newcommand{\N}{\mathrm{N}}

\newcommand{\kp}{\mathfrak{p}}
\newcommand{\kP}{\mathfrak{P}}

\newcommand{\cO}{\mathcal{O}}
\newcommand{\ord}{\mathrm{ord}\,}

\newcommand{\cQ}{\mathcal{Q}}

\renewcommand{\Re}{\mathrm{Re}}

\newcommand{\stab}{{\rm stab}}

\newcommand{\sumS}{\sideset{}{^\star}\sum}
\newcommand{\sumD}{\sideset{}{^\dagger}\sum}
\newcommand{\SL}{ {\rm SL}}

\renewcommand{\pmod}[1]{\left(\mathrm{mod}\,\,#1\right)}

\newconstantfamily{abcon}{symbol=c}
\newconstantfamily{abcon2}{symbol=C}

\title{A unified and improved Chebotarev density theorem}
\author{Jesse Thorner}
\address{Department of Mathematics, Stanford University, Stanford, CA 94305}
\email{jthorner@stanford.edu}
\author{Asif Zaman}
\address{Department of Mathematics, Stanford University, Stanford, CA 94305}
\email{aazaman@stanford.edu}

\begin{document}

\maketitle

\section{Introduction and statement of results}
\label{sec:intro}

\subsection{Introduction}
Let $L/F$ be a Galois extension of number fields with Galois group $G$.  For each prime ideal $\kp$ of $F$ that is unramified in $L$, we use the Artin symbol $[\frac{L/F}{\kp}]$ to denote the conjugacy class of $G$ consisting of the set of Frobenius automorphisms attached to the prime ideals $\kP$ of $L$ which lie over $\kp$.  For any conjugacy class $C \subseteq G$, define the function
\begin{equation}
\label{def:pi_C}
\pi_C(x)=\pi_C(x,L/F)=\#\Big\{\N_{F/\Q}\kp\leq x\colon \textup{$\kp$ unramified in $L$, $\Big[\frac{L/F}{\kp}\Big]=C$}\Big\},
\end{equation}
where $\N_{F/\Q}$ is the absolute norm of $F/\Q$.  The Chebotarev density theorem states that
\[
\pi_C(x)\sim \frac{|C|}{|G|}\mathrm{Li}(x)\qquad \textup{as $x\to\infty$.}
\]

It follows from work of V.K. Murty \cite[Section 4]{VKM} that there exists an absolute,  effective, and positive constant $\Cl[abcon]{CDT}$ such that
\begin{equation}
\label{eqn:CDT}
\pi_C(x)=\frac{|C|}{|G|}\Big(\mathrm{Li}(x)-\theta_1 \mathrm{Li}(x^{\beta_1})+O\Big(xe^{-\Cr{CDT}\sqrt{\frac{\log x}{n_L}}}\Big) \Big),\quad \log x\gg \frac{(\log D_L)^2}{n_L} + n_L(\log n_L)^2,
\end{equation}
which refines a well-known result of Lagarias and Odlyzko \cite[Theorem 1.2]{LO}.  Here, $D_L$ is the absolute discriminant of $L$, $n_L=[L:\Q]$ is the degree of $L$ over $\Q$, $\beta_1$ is a possible Landau-Siegel zero of the Dedekind zeta function $\zeta_L(s)$ of $L$, and $\theta_{1} = \theta_1(C)\in \{-1,0,1\}$ depends on $C$; in particular, $\theta_{1}(C)=0$ if and only if $\beta_1$ does not exist.   For comparison, Lagarias and Odlyzko \cite[Theorem 1.1]{LO} proved that the generalized Riemann hypothesis for $\zeta_L(s)$ implies the more uniform result
\begin{equation}
\label{eqn:CDT_GRH}
\pi_C(x)=\frac{|C|}{|G|}\Big(\mathrm{Li}(x)+O(\sqrt{x}\log (D_L x^{n_L})) \Big),\qquad x\gg (\log D_L)^2 (\log\log D_L)^4.
\end{equation}

As of now, the best bound for $\beta_1$ is due to Stark \cite[Theorem 1', p. 148]{Stark-1974}; it implies that
\begin{equation}
1-\beta_1 \gg (n_L^{n_L} \log D_L + D_L^{1/n_L})^{-1}. 
\label{eqn:Stark}
\end{equation}
Therefore, in order to ensure that $\frac{|C|}{|G|}\mathrm{Li}(x)$ dominates all other terms in \eqref{eqn:CDT}, one must therefore take the range of $x$ to be
\begin{equation}
\label{eqn:bad_range}
\log x \gg n_L^{-1}(\log D_L)^2 + n_L (\log n_L)^2 + (1-\beta_1)^{-1}	
\end{equation}
and apply \eqref{eqn:Stark} if $\beta_1$ exists. Otherwise, one omits the last term in \eqref{eqn:bad_range} if  $\beta_1$ does not exist. Regardless, \eqref{eqn:bad_range} is very prohibitive in many applications where uniformity in $L/F$ is crucial.  Thus it often helps in applications to have upper and lower bounds for $\pi_C(x)$ of order $\mathrm{Li}(x)$ in ranges of $x$ which are more commensurate with \eqref{eqn:CDT_GRH}.  Lagarias, Montgomery, and Odlyzko \cite{LMO} made substantial progress on these problems; their work has been improved upon by Weiss \cite{Weiss}, the authors \cite{TZ1,TZ2}, and Zaman \cite{Zaman_Thesis}.  In particular, it follows from the joint work of the authors \cite{TZ1,TZ2} that there exist absolute, effective constants $A>2$ and $B>2$ such that if $D_L$ is sufficiently large, then
\begin{equation}
	\label{eqn:LMO_Linnik}
	\frac{1}{(D_L n_L^{n_L})^A}\frac{|C|}{|G|}\mathrm{Li}(x)\ll \pi_C(x)<(2+o(1))\frac{|C|}{|G|}\mathrm{Li}(x) \qquad\textup{for $x \geq (D_L n_L^{n_L})^B$,}
\end{equation}
where the $o(1)$ term tends to zero as $(\log x)/\log(D_L n_L^{n_L})$ tends to infinity\footnote{The term $n_L^{n_L}$ is usually negligible compared to a power of $D_L$. If not, one might appeal to \cite[Theorem 1.3.1]{Zaman_Thesis} which states that $\pi_C(x)\gg D_L^{-A}\frac{|C|}{|G|}\mathrm{Li}(x)$ for $x\geq D_L^B$. }.  

To summarize the above discussion, suppose that we are in the worst case scenario with $\theta_1=1$ and $\beta_1$ is as bad as \eqref{eqn:Stark} permits.  If one is willing to sacrifice an asymptotic equality for $\pi_C(x)$ in order to obtain estimates in noticeably better ranges than \eqref{eqn:bad_range}, then one might use \eqref{eqn:LMO_Linnik}.  On the other hand, if one needs an asymptotic equality for $\pi_C(x)$, then one uses \eqref{eqn:CDT} in the prohibitive range \eqref{eqn:bad_range}.

\subsection{Results}
Our main result, \cref{thm:main_theorem}, is a new asymptotic equality for $\pi_C(x)$ which interpolates both of the aforementioned options while providing several new options.  In other words, we prove a new asymptotic equality for $\pi_C(x)$ from which one may deduce both \eqref{eqn:CDT} and \eqref{eqn:LMO_Linnik}.  First, we present a simplified version of the main result.
\begin{thm}
	\label{thm:CDT_1}
	Let $L/F$ be a Galois extension of number fields with Galois group $G$, and let $C\subseteq G$ be a conjugacy class.  Let $\beta_1$ denote the Landau-Siegel zero of the Dedekind zeta function $\zeta_L(s)$, if it exists.  There exist absolute and effective constants $\Cl[abcon]{1}>0$ and $\Cl[abcon]{2}>0$ such that if $L\neq \Q$ and $x\geq(D_L n_L^{n_L})^{\Cr{1}}$, then
	\[
	\pi_C(x)=\frac{|C|}{|G|}\Big(\mathrm{Li}(x)-\theta_1 \mathrm{Li}(x^{\beta_1}) \Big)\Big(1+O\Big(\exp\Big[-\frac{\Cr{2}\log x}{\log(D_L n_L^{n_L})}\Big]+\exp\Big[-\frac{(\Cr{2}\log x)^{1/2}}{n_L^{1/2}}\Big]\Big)\Big),
	\]
	where $\theta_1 = \theta_1(C) \in \{-1,0,1\}$. In particular, $\theta_1 =0$ precisely when $\beta_1$ does not exist. 
\end{thm}

By verifying when
\[
\exp\Big[-\frac{(\Cr{2}\log x)^{1/2}}{n_L^{1/2}}\Big]\gg \exp\Big[-\frac{\Cr{2}\log x}{\log(D_L n_L^{n_L})}\Big],
\]
we see that \cref{thm:CDT_1} recovers \eqref{eqn:CDT} and is therefore a uniform improvement over it.  Also, it follows from the mean value theorem and \eqref{eqn:Stark} that
\begin{equation}
\label{eqn:2_stark}
\Li(x) - \theta_1 \Li(x^{\beta_1}) \gg \big( (1-\beta_1) \log(D_L n_L^{n_L}) \big) \mathrm{Li}(x)\gg \frac{\log(D_L n_L^{n_L})}{D_L^{1/n_L} + n_L^{n_L} \log D_L}\mathrm{Li}(x).
\end{equation}
With this lower bound at our disposal, one can see that \cref{thm:CDT_1} recovers \eqref{eqn:LMO_Linnik}.  Thus \cref{thm:CDT_1} unifies and improves both \eqref{eqn:CDT} and \eqref{eqn:LMO_Linnik}.

As noted above, if one wants $\frac{|C|}{|G|}\mathrm{Li}(x)$ to dominate all other terms in \eqref{eqn:CDT}, then one must take $x$ in the range \eqref{eqn:bad_range}.  However,  one can plainly see that
\begin{equation}
\label{eqn:MT_awkward}
\frac{|C|}{|G|}(\mathrm{Li}(x)-\theta_1\mathrm{Li}(x^{\beta_1}))
\end{equation}
dominates all other terms in \cref{thm:CDT_1} for all $x$ in the claimed range, provided that $\Cr{1}$ is suitably large compared to $\Cr{2}$.  At first glance, it may seem awkward that we adjoin the contribution from $\beta_1$ to the ``main term'' when it is classically viewed as an error term.  But without eliminating the existence of $\beta_1$, it is well-known that in situations where $\theta_1\neq0$ and $x$ is small, say $\log x\ll \log(D_L n_L^{n_L})$, the term $-\theta_1\frac{|C|}{|G|}\mathrm{Li}(x^{\beta_1})$ is more properly treated as a secondary term than an error term.  When $\theta_1=1$ and $\beta_1$ is especially close to 1, this secondary term causes serious difficulties in the proof of Linnik's bound for the least prime in an arithmetic progression \cite{Linnik}.  Fortunately, it follows from \eqref{eqn:2_stark} that regardless of whether $\beta_1$ exists, we have
\begin{equation}
\label{eqn:upperlower}
\mathrm{Li}(x)\ll_{L}\mathrm{Li}(x)-\theta_1\mathrm{Li}(x^{\beta_1})< 2\mathrm{Li}(x).
\end{equation}
Therefore, in the range of $x$ where $-\frac{|C|}{|G|}\theta_1\mathrm{Li}(x^{\beta_1})$ acts like a secondary term, \eqref{eqn:upperlower} shows that \cref{thm:CDT_1} recovers upper and lower bounds of order $\mathrm{Li}(x)$ precisely because \eqref{eqn:MT_awkward} dominates all other terms in \cref{thm:CDT_1}.  This perspective is implicit in Linnik's work.  On the other hand, when $x$ is sufficiently large in terms of $L/F$ per \eqref{eqn:bad_range}, the contribution from $\beta_1$ can be safely absorbed into the $O$-term in \cref{thm:CDT_1}.  In light of these observations, we believe that viewing \eqref{eqn:MT_awkward} as the ``main term'' in \cref{thm:CDT_1} helps to clarify the role of the contribution from $\beta_1$ when one transitions from small values of $x$ to large values of $x$.

Upon considering the $O$-term in \cref{thm:CDT_1}, we see that \cref{thm:CDT_1} noticeably improves the range of $x$ in which we have an asymptotic equality for $\pi_C(x)$.

\begin{cor}
	\label{cor:CDT_2}
	If $\frac{\log x}{\log(D_L n_L^{n_L})}\to\infty$, then $\pi_C(x) \sim\frac{|C|}{|G|}(\mathrm{Li}(x)-\theta_{1}\mathrm{Li}(x^{\beta_1}))$.
\end{cor}

\cref{thm:CDT_1} also produces a new asymptotic equality in which the error term saves an arbitrarily large power of $\log x$ in a much stronger range of $x$ than \eqref{eqn:CDT}.  

\begin{cor}
	\label{cor:CDT_3}
	Let $A > 1$. If $\log x \gg_A (\log D_L)(\log\log D_L) + n_L (\log n_L)^2$, then
	\begin{equation}
	\pi_C(x) = \frac{|C|}{|G|}\Big(\mathrm{Li}(x)-\theta_{1}\mathrm{Li}(x^{\beta_1}) \Big) \Big(1 + O_A\big((\log x)^{-A}\big) \Big). 
	\label{eqn:CDT_logs}
\end{equation}
\end{cor}

In order to state the main result from which \cref{thm:CDT_1} follows, we introduce some additional notation.  Let $H\subseteq G$ be an abelian subgroup of $G$ such that $H\cap C$ is nonempty, and let $K=L^H$ be the fixed field of $H$.  The characters $\chi$ in the dual group $\widehat{H}$ are Hecke characters; we write the conductor of $\chi$ as $\mathfrak{f}_{\chi}$.  Define
\begin{equation}
\mathcal{Q}=\mathcal{Q}(L/K)=\max_{\chi\in\widehat{H}}\N_{K/\Q}\mathfrak{f}_{\chi}.
\label{eqn:MaxConductor_Abelian}
\end{equation}
We write the $L$-function associated to such a Hecke character as $L(s,\chi,L/K)$. From work of Stark \cite{Stark-1974}, at most one real Hecke character $\chi_1\in\widehat{H}$ has an associated Hecke $L$-function $L(s,\chi_1,L/K)$ with a Landau-Siegel zero $\beta_1 = 1 - \lambda_1/\log(D_K\mathcal{Q} n_K^{n_K})$, where $0<\lambda_1 < \frac{1}{8}$.

\begin{thm}
\label{thm:main_theorem}
	Let $L/F$ be a Galois extension of number fields with Galois group $G$, and let $C\subseteq G$ be a conjugacy class.  Let $H\subseteq G$ be an abelian subgroup such that $C\cap H$ is nonempty, let $K$ be the fixed field of $H$, and choose $g_C\in C\cap H$.  If $x\geq (D_K \cQ n_K^{n_K})^{\Cr{1}}$, then
	\begin{align*}
	\pi_C(x)=\frac{|C|}{|G|}\Big(\mathrm{Li}(x)-\theta_1\mathrm{Li}(x^{\beta_1})\Big)\Big(1+O\Big(\exp\Big[-\frac{\Cr{2}\log x}{\log(D_K\mathcal{Q}n_K^{n_K})}\Big]+\exp\Big[-\frac{(\Cr{2}\log x)^{1/2}}{n_K^{1/2}}\Big]\Big)\Big),
	\end{align*}
	where $\theta_1=\chi_1(g_C)$ if $\beta_1$ exists and $\theta_1=0$ otherwise and $\mathcal{Q}$ is given by \eqref{eqn:MaxConductor_Abelian}.  The constants $\Cr{1}$ and $\Cr{2}$ are the same as in \cref{thm:CDT_1}.
	\end{thm}
\begin{remark}
As a group-theoretic quantity, $\theta_1$ depends on the choice of $g_C\in C\cap H$.  However, if $\theta_1\neq0$, then the existence of $\beta_1$ implies that $\theta_1$ is well-defined. 
\end{remark}

\subsection{An application}
\label{subsec:application}
While it is aesthetically appealing to be able to encapsulate the work in \cite{LMO,LO,VKM,TZ1,TZ2,Weiss} with a single asymptotic equality, \cref{thm:main_theorem} can make progress in certain sieve-theoretic problems when one must compute the local densities. As an example, we prove a new result in the study of primes represented by binary quadratic forms. Let 
\[
f(u,v) = au^2 + buv + cv^2 \in\Z[u,v]
\]
be a positive definite binary quadratic form of discriminant $D = b^2 - 4ac<0$. We do not assume that $D$ is fundamental. The group $\SL_2(\Z)$ naturally acts on such forms by $(T \cdot f)(\mathbf{x}) = f(T \mathbf{x})$ for $T \in \SL_2(\Z)$.  The class number $h(D)$ is the number of such forms up to $\SL_2$-equivalence.  If $f$ is primitive (that is, $\gcd(a,b,c) = 1$) then it is a classical consequence of the Chebotarev density theorem and class field theory that  
\begin{equation}
\label{eqn:BQF_CDT_asymp}
\frac{1}{|\mathrm{stab}(f)|}\dsum_{\substack{u,v \in \Z \\ au^2 + buv + cv^2 \leq x}} \mathbf{1}_{\mathbb{P}}(au^2 + buv + cv^2) \sim \frac{\mathrm{Li}(x)}{h(D)} \qquad \text{as $x \rightarrow \infty$,}
\end{equation}
where $\mathbf{1}_{\mathbb{P}}$ is the indicator function for the odd primes and
\[
\mathrm{stab}(f) = \{ T \in \SL_2(\Z) : T \cdot f = f \}.
\]
Note $|\stab(f)| = 2$ unless $D = -3$ or $-4$ in which case it equals $6$ and $4$ respectively.

We consider the question of imposing restrictions on the integers $u$ and $v$ which comprise a solution to the equation $p=f(u,v)$.  In the special case of $f(u,v)=u^2+v^2$, Fouvry and Iwaniec \cite{MR1438827} proved that there are infinitely many primes $p$ such that $p=u^2+v^2$ and $u$ is prime. Their proof, which relies on the circle method, enables them to asymptotically count such primes.

One might ask whether their methods extend to all positive definite primitive $f(u,v)$ with strong uniformity in the discriminant $D$. The answer is not clear to the authors.  Nevertheless, \cref{thm:main_theorem} enables us to study the distribution of primes $p=f(u,v)$ with some control over the divisors of $u$ and $v$ while maintaining strong uniformity in $D$.  We prove the following result in \cref{sec:BQF}.
\begin{thm}
\label{thm:application}
Let $D \leq -3$ be an integer and let $f(u,v) = au^2 + buv + cv^2$ be a positive definite primitive integral binary quadratic form with discriminant $D = b^2 - 4ac$.  Let $P$ be any integer dividing the product of primes $p \leq z$.  For all $A\geq 1$, there exists a constant $\eta=\eta(A)>0$ such that if $3 \leq z \leq x^{\eta/\log\log x}$ and $3 \leq |D|\leq x^{\eta/\log\log z}$, then
\begin{equation}
\label{eqn:BQF_primes}
\frac{1}{|\stab(f)|} \dsum_{\substack{u,v \in \Z \\ au^2 + buv + cv^2 \leq x \\ \gcd(uv,P)=1}} \mathbf{1}_{\mathbb{P}}(au^2 + buv + cv^2)= \delta_{f}(P) \frac{\Li(x) - \Li(x^{\beta_1})}{h(D)} \{1 + O_A((\log z)^{-A})\}.
\end{equation}
Here, $\beta_1$ is a real simple zero of the Dedekind zeta function $\zeta_{\Q(\sqrt{D})}(s)$ (if it exists),
\begin{equation}
\delta_{f}(P) =  \prod_{\substack{p \mid P}} \Big(1 - \frac{2 - \mathbf{1}_{p \mid a}(p) - \mathbf{1}_{p \mid c}(p)}{p - (\frac{D}{p})} \Big),
\label{eqn:bqf_euler_product}
\end{equation}
$(\frac{D}{p})$ is the Legendre symbol for $p \neq 2$, $(\frac{D}{2})$ is defined by \eqref{def:Legendre_2}, and the term $\Li(x^{\beta_1})$ is omitted if $\beta_1$ does not exist. 
\end{thm} 
\begin{remark}
	The constant $\delta(P)$ is always non-negative. It is possible that $\delta_f(P) = 0$ due to the local factor at $p=2$ in the product but this occurs precisely when the form $f(u,v)$ does not represent any odd primes. Since $\mathbf{1}_{\mathbb{P}}$ is the indicator function for the {\it odd} primes, \eqref{eqn:BQF_primes} trivially holds in this case. The details of this casework are verified in \cref{subsec:bqf_main_term}. 
\end{remark}

While it is natural to think of $P$ as equal to the product of primes up to $z$, we immediately obtain from \cref{thm:application} the following corollary when $P$ is a {\it fixed} divisor of the product of primes up to $z$ and $z \rightarrow \infty$ arbitrarily slowly.

\begin{cor}
\label{cor:application}
	Keep the assumptions of \cref{thm:application}. If the integer $P \geq 1$ is fixed, then 
	\[
	\frac{1}{|\stab(f)|} \dsum_{\substack{u,v \in \Z \\ au^2 + buv + cv^2 \leq x \\ \gcd(uv,P)=1}} \mathbf{1}_{\mathbb{P}}(au^2 + buv + cv^2) \sim \delta_{f}(P) \frac{\Li(x) - \Li(x^{\beta_1})}{h(D)}  \quad \text{as $\frac{\log x}{\log |D|} \rightarrow \infty$}. 
	\]
	In particular, there exists a prime $p \leq |D|^{\alpha}$ and $u,v \in \Z$ such that $p =f(u,v)$, $p\nmid D$, and $\gcd(uv,P) = 1$, where $\alpha=\alpha(P)>0$ is a sufficiently large constant depending only on $P$.
\end{cor}

In order to prove \cref{thm:application} with strong uniformity in $z$ and $|D|$, one needs asymptotic control over sums like \eqref{eqn:BQF_CDT_asymp} (see \eqref{def:bqf_congruence_sum} below) when $x$ is as small as a polynomial in the discriminant, regardless of whether $\zeta_{\Q(\sqrt{D})}(s)$ has a Landau-Siegel zero.  This is precisely what \cref{thm:main_theorem} provides.  For comparison, a slightly stronger version of \eqref{eqn:CDT} that follows from  \cite{VKM} along with the effective bound $(1-\beta_1)^{-1}\ll |D|^{1/2}\log|D|$ can produce \eqref{eqn:BQF_primes} with the inferior ranges 
\[
3\leq |D|\ll_{\epsilon}(\log x)^{2}/(\log\log x)^2 \qquad \text{and} \qquad  3\leq z\leq \exp(c\sqrt{\log x})
\] 
where $c>0$ is an absolute constant and $\epsilon>0$. As one can plainly see, \cref{thm:main_theorem} yields substantial gains over earlier versions of the Chebotarev density theorem. See \cref{remark:bqf_cdt} for further discussion. 

\subsection{Overview of the methods}
We now give an overview of how the proof of \cref{thm:main_theorem} differs from the proofs in \cite{LMO,LO,VKM,TZ1,TZ2,Weiss}.  For convenience, we refer to
\[
\frac{|C|}{|G|}\Big(\mathrm{Li}(x)-\theta_1\mathrm{Li}(x^{\beta_1})\Big)
\]
as the ``main term'' in \cref{thm:main_theorem} and all other terms   as the ``error term''.

The key difference between the proof of \eqref{eqn:CDT} and the proof of \cref{thm:main_theorem} lies in the study of the non-trivial low-lying zeros of $\zeta_L(s)$.  The standard zero-free region for $\zeta_L(s)$ indicates that the low-lying zeros of $\zeta_L(s)$ lie further away from the edge of the critical strip $\{s\in\mathbb{C}\colon 0<\Re(s)<1\}$ than zeros of large height.  However, the treatments in \cite{LO,VKM} handle the contribution from the all of the non-trivial zeros by assuming that the low-lying zeros (other than $\beta_1$, if it exists) lie just as close to the edge of the critical strip as zeros of large height.  This unduly inflates the contribution from the low-lying zeros, leading to the poor field uniformity in \eqref{eqn:CDT} along with the poor dependence on the Landau-Siegel zero $\beta_1$ if it exists.  Consequently, both the range of $x$ and the quality the error term in \eqref{eqn:CDT} directly depend on the quality of zero-free region available for $\zeta_L(s)$.

In order to efficiently handle the contribution to $\pi_C(x)$ which arises from the low-lying zeros of $\zeta_L(s)$, we factor $\zeta_L(s)$ as a product of Hecke $L$-functions associated to the Hecke characters of the abelian extension $L/K$ and apply a log-free zero density estimate and the zero repulsion phenomenon for these $L$-functions. As in  Linnik's work on arithmetic progessions, one typically uses these tools to establish upper and lower bounds of $\pi_C(x)$ when $x$ is small instead of asymptotic equalities \cite{TZ1,TZ2,Weiss}.  In order to facilitate the analysis involving the log-free zero density estimate, we weigh the contribution of each prime ideal counted by $\pi_C(x)$ with a weight whose Mellin transform has carefully chosen decay properties (\cref{lem:WeightChoice}).  Similar variations are a critical component in the proofs of \eqref{eqn:LMO_Linnik} in \cite{TZ1,TZ2,Weiss}.

By using a log-free zero density estimate and the zero repulsion phenomenon, we ensure that the main term in \cref{thm:main_theorem} \textit{always} dominates the error term in \cref{thm:main_theorem} when $x$ is at least a polynomial in $D_K\mathcal{Q}n_{K}^{n_{K}}$, regardless of whether $\beta_1$ exists.  As one can see from the ensuing analysis, the quality of the zero-free region dictates the quality of the error term but has no direct impact on the valid range of $x$. This ``decoupling'' feature contrasts with the proof of \eqref{eqn:CDT}, where the quality of the zero-free region simultaneously determines both the quality of the error term and the range of $x$ in which the main term dominates.

After we ``decouple'' the range of $x$ from the influence of the zero-free region, we are finally prepared to separate the contribution of the low-lying zeros from the contribution of the zeros with large height using a dyadic decomposition. This leads to savings over \eqref{eqn:CDT} only because we have already ensured via the log-free zero density estimate and zero repulsion that the main term in \cref{thm:main_theorem} dominate the error term regardless of whether $\beta_1$ exists. An additional benefit of this argument is an expression for the error term in \cref{thm:main_theorem} as a straightforward single-variable optimization problem involving $x$ and the zero-free region (\cref{lem:SumOverZeros} and \eqref{eqn:WeightPrimeSum_Penultimate}). This simplification allows us to easily determine the error term with complete uniformity in $D_K$, $[K:\Q]$, $\mathcal{Q}$, and $x$ (\cref{lem:ErrorOptimization_Classical}).

The fact that \cref{thm:main_theorem} holds for {\it all} Galois extensions $L/F$ is a fairly subtle matter.  In the case where $F=\Q$ and $L/\Q$ is a cyclotomic extension,  the Chebotarev density theorem reduces to the prime number theorem for arithmetic progressions.  Stark's bound for $\beta_1$ (\cref{thm:Stark}, a refinement of \eqref{eqn:Stark}) recovers a lower bound for $1-\beta_1$ which is commensurate with the lower bound for $1-\beta_1$ that follows from Dirichlet's analytic class number formula for cyclotomic extensions; this suffices for our purposes.  In the cyclotomic setting, our proofs only need to quantify the zero repulsion from a Landau-Siegel zero with a strong zero-free region for low-lying zeros (\cref{thm:DH} with $t\leq 4$). However, if $L/F$ is a Galois extension where the root discriminant of $L$ is especially small, which can happen in infinite class field towers, then Stark's lower bound for $1-\beta_1$ is quite small.  In this case, the approach which worked well for cyclotomic extensions of $\Q$ appears insufficient to prove \cref{thm:CDT_1} for all $x$ in our claimed range.

To address this problem, we use a log-free zero density estimate for Hecke $L$-functions that naturally incorporates the zero repulsion phenomenon.  Roughly speaking, when $\beta_1$ is especially close to 1, the quality of the log-free zero density estimate improves by a factor of $1-\beta_1$; this is stronger than the classical formulation of the zero repulsion phenomenon.  Therefore, if $1-\beta_1$ happens to be as small as Stark's lower bound allows, the quality of the log-free zero density estimate increases dramatically.  This offsets the adverse effect of $\beta_1$ in the small root discriminant case.  The idea of incorporating the zero repulsion phenomenon directly into the log-free zero density estimate goes back to Bombieri \cite{Bombieri2} in the case of Dirichlet characters.  For Hecke $L$-functions over number fields, this was first proved by Weiss (see \cref{thm:LFZDE} below). The details of this obstacle and why we genuinely need the particular log-free zero density estimate in \cref{thm:LFZDE} are contained in \cref{sec:appendix}, especially \cref{rem:SZ}.

\subsection*{Acknowledgements}
We thank Kannan Soundararajan for helpful discussions and the anonymous referee for providing very thorough comments on our initial submission.  Jesse Thorner is partially supported by a NSF Mathematical Sciences Postdoctoral Fellowship, and Asif Zaman is partially supported by a NSERC Postdoctoral Fellowship.

\section{Setup and notation}
\label{sec:Setup}
 
Throughout the paper, let $c_1$, $c_2$, $c_3,\ldots$ be a sequence of absolute, effective, and positive constants.  All implied constants in the inequalities $f\ll g$ and $f=O(g)$ are absolute and effective unless noted otherwise.

Recall $F$ is a number field with ring of integers $\cO_F$, absolute norm $\N = \N_{F/\Q}$, absolute discriminant $D_F = |\mathrm{disc}(F/\Q)|$, and degree $n_F = [F:\Q]$. Integral ideals will be denoted by $\kn$ and prime ideals by $\kp$.  Moreover, $L/F$ is a Galois extension of number fields with  Galois group $G = \mathrm{Gal}(L/F)$. For prime ideals $\kp$ of $F$ unramified in $L$, the Artin symbol $[\frac{L/F}{\kp}]$ is the conjugacy class of Frobenius automorphisms of $G$ associated to prime ideals $\mathfrak{P}$ of $L$ lying above $\kp$.

\subsection{Prime counting functions}  For a conjugacy class $C$ of $G$ and $x \geq 2$, let $\pi_C(x)$ be as in \eqref{def:pi_C} and define
	\begin{equation}
	\psi_C(x) = \psi_C(x,L/F) = \frac{|C|}{|G|}\sum_{\psi} \bar{\psi}(C)  \frac{1}{2\pi i} \int_{2-i\infty}^{2+i\infty} -\frac{L'}{L}(s,\psi,L/F) \frac{x^s}{s} ds,
	\label{def:psi_C_integral}	
	\end{equation}
	where $\psi$ runs over the irreducible Artin characters of $G = \mathrm{Gal}(L/F)$ and $L(s,\psi,L/F)$ is the Artin $L$-function of $\psi$. It follows from Mellin inversion \cite[p.283]{LMO} that
	\begin{equation}
		\psi_C(x) = \sum_{\N\kn \leq x} \Lambda_F(\kn) \mathbf{1}_C(\kn), \\
		\label{def:psi_C}
	\end{equation}
	where
	\begin{equation}
		\Lambda_F(\kn) = \begin{cases} \log \N\kp & \text{if $\kn=\kp^j$ for some prime ideal $\kp$ and some integer $j\geq1$,} \\ 0 & \text{otherwise}.	
 	\end{cases}
	\end{equation}
	Here, $0 \leq \mathbf{1}_C(\kn) \leq 1$ for all ideals $\kn$ and for prime ideals $\kp$ unramified in $L$ and $j \geq 1$,
	\begin{equation}
		\mathbf{1}_C(\kp^j) = \begin{cases} 
 					1 & \text{if $[\frac{L/F}{\kp}]^j \subseteq C$,} \\
 					0 & \text{otherwise}.
				 \end{cases} 
	\end{equation}
 	The prime counting functions $\pi_C$ and $\psi_C$ are related via partial summation.
	\begin{lem} \label{lem:Pi_to_Psi}
		For $x \geq 2$,
		\[
		\pi_C(x) = \frac{\psi_C(x)}{\log x} + \int_{\sqrt{x}}^x \frac{\psi_C(t)}{t (\log t)^2} dt + O\Big(\log D_L + \frac{n_F x^{1/2}}{\log x} \Big). 
		\]	
	\end{lem}
	\begin{proof}
		Note the norm of the product of ramified prime ideals divides $D_L$ and the number of prime ideals $\kp$ with norm equal to a given rational prime $p$ is at most $n_F$. Thus, 
		\[
		\pi_C(x) = \sum_{\sqrt{x} < \N\kp \leq x} \mathbf{1}_C(\kp)  + O\Big(\frac{n_F x^{1/2}}{\log x} + \log D_L\Big). 
		\]
		Define $\theta_C(x) = \sum_{\N\kp \leq x} \mathbf{1}_C(\kp) \log \N\kp$. It follows by partial summation as well as the previous observations that 
		\[
		\sum_{\sqrt{x} < \N\kp \leq x} \mathbf{1}_C(\kp) = \int_{\sqrt{x}}^x \frac{\theta_C(t)}{t(\log t)^2} dt + \frac{\theta_C(x)}{\log x}
		\]
		Finally, one can verify that $|\theta_C(x) - \psi_C(x)|\ll n_F x^{1/2}$ by trivially estimating the number of prime ideal powers with norm at most $x$. Collecting all of these estimates yields the lemma.
	\end{proof}
	
	\subsection{Choice of weight}
We now define a weight function which will be used to count prime ideals with norm between $\sqrt{x}$ and $x$.

\begin{lem}
\label{lem:WeightChoice}

Choose $x \geq 3$, $\epsilon \in (0,1/4)$, and a positive integer $\ell \geq 1$.  Define $A = \epsilon/(2 \ell \log x)$.  There exists a continuous function $f(t)  = f(t; x, \ell, \epsilon)$ of a real variable $t$ such that:
\begin{enumerate}[(i)]
	\item $0 \leq f(t) \leq 1$ for all $t \in \R$, and $f(t) \equiv 1$ for $\tfrac{1}{2} \leq t \leq 1$.
	\item The support of $f$ is contained in the interval $[\tfrac{1}{2} - \frac{\epsilon}{\log x}, 1 +  \frac{\epsilon}{\log x}]$. 
	\item Its Laplace transform $F(z) = \int_{\R} f(t) e^{-zt}dt$ is entire and is given by
			\begin{equation}	
				F(z) = e^{-(1+ 2\ell A)z} \cdot \Big( \frac{1-e^{(\frac{1}{2}+2\ell A)z}}{-z} \Big) \Big( \frac{1-e^{2Az}}{-2Az} \Big)^{\ell}.
				\label{eqn:WeightLaplace}
			\end{equation}
	\item Let $s = \sigma + i t, \sigma > 0, t \in \R$ and $\alpha$ be any real number satisfying $0 \leq \alpha \leq \ell$. Then
	\[
	|F(-s\log x)|\leq 
		\displaystyle\frac{e^{\sigma \epsilon} x^{\sigma}}{|s| \log x} \cdot \big( 1 + x^{-\sigma/2} \big) \cdot  \Big( \frac{2\ell}{\epsilon|s|} \Big)^{\alpha}. 
	\]
	Moreover, $|F(-s\log x)| \leq e^{\sigma \epsilon} x^{\sigma}$ and $1/2 < F(0) < 3/4$. 
	\item If $\tfrac{3}{4} < \sigma \leq 1$ and $x \geq 10$, then
		  \begin{equation}
		F(-\log x) \pm F(-\sigma\log x) = \Big( \frac{x}{\log x} \pm   \frac{x^{\sigma}}{\sigma \log x} \Big) \big\{ 1 + O(\epsilon) \big\} + O\Big(\frac{x^{1/2}}{\log x}\Big). 
		\label{eqn:WeightChoice_MainTerm} 
		  \end{equation}
	\item Let $s = -\tfrac{1}{2}+it$ with $t \in \R$. Then
	\[
	|F(-s\log x)| \leq  \frac{5 x^{-1/4}}{\log x} \Big( \frac{2\ell}{\epsilon}\Big)^{\ell} (1/4+t^2)^{-\ell/2}.
	\]
\end{enumerate}
\end{lem}
\begin{proof}
These are the contents of \cite[Lemma 2.2]{TZ2} except for \eqref{eqn:WeightChoice_MainTerm}, which we now prove. Let $\frac{3}{4}<\sigma\leq 1$.  From (iii), we observe that
\begin{equation}
F(-\sigma \log x) = \frac{x^{\sigma}}{\sigma \log x} \Big( \frac{e^{\epsilon\sigma/\ell}-1}{\epsilon\sigma/\ell}\Big)^{\ell} + O\Big(\frac{x^{\sigma/2}}{\sigma \log x}\Big).
\label{eqn:WeightChoice_Real}
\end{equation}
The two cases of $F(-\log x)\pm F(-\sigma \log x)$ are proved differently; we first handle the $+$ case.  It follows from \eqref{eqn:WeightChoice_Real} that
\[
F(-\log x)+F(-\sigma \log x)=\frac{x}{\log x} \Big( \frac{e^{\epsilon/\ell}-1}{\epsilon/\ell}\Big)^{\ell}+\frac{x^{\sigma}}{\sigma \log x} \Big( \frac{e^{\epsilon\sigma/\ell}-1}{\epsilon\sigma/\ell}\Big)^{\ell} + O\Big(\frac{x^{\sigma/2}}{\sigma \log x}\Big).
\]
The desired asymptotic for $F(-\log x)+F(-\sigma\log x)$ now follows from the Taylor series expansion
\[
\Big( \frac{ e^{\epsilon\sigma/\ell}-1}{\epsilon\sigma/\ell}\Big)^{\ell} = 1 + O(\sigma \epsilon),
\]
which is valid for $0<\sigma\leq 1$.

For the case of $F(-\log x)- F(-\sigma \log x)$, we first observe that \eqref{eqn:WeightChoice_Real} implies
\begin{equation}
\begin{aligned}
(\log x) ( F(-\log x) - F(-\sigma \log x) ) 
	& = x \Big( \frac{e^{\epsilon/\ell}-1}{\epsilon/\ell}\Big)^{\ell}  - \frac{x^{\sigma}}{\sigma} \Big( \frac{e^{\epsilon\sigma/\ell}-1}{\epsilon\sigma/\ell}\Big)^{\ell}  + O(x^{1/2}).
\end{aligned}
\label{eqn:WeightChoice_MainTerm_Diff}
\end{equation}
Set
\[
a= \frac{e^{\epsilon/\ell}-1}{\epsilon/\ell},\qquad b=\frac{e^{\sigma\epsilon/\ell}-1}{\sigma\epsilon/\ell}
\]
so that $a > b \geq 1$.  With this convention, we rewrite \eqref{eqn:WeightChoice_MainTerm_Diff} as
\begin{equation}
\label{eqn:blue_harvest_1}
(\log x) ( F(-\log x) - F(-\sigma \log x) ) =x a^{\ell}-\frac{x^{\sigma}}{\sigma}b^{\ell}+O(x^{1/2}).
\end{equation}
Since $a>b\geq 1$, it follows from the bound $a^{\ell}-b^{\ell}\ll(a-b)\cdot \ell a^{\ell}$ that
\begin{equation}
\label{eqn:blue_harvest_2}
x a^{\ell}-\frac{x^{\sigma}}{\sigma}b^{\ell}=\Big(x -\frac{x^{\sigma}}{\sigma}\Big)a^{\ell}+\frac{x^{\sigma}}{\sigma}(a^{\ell}-b^{\ell})=\Big(x-\frac{x^{\sigma}}{\sigma}\Big)a^{\ell}+O\Big(\frac{x^{\sigma}}{\sigma}(a-b)\ell a^{\ell}\Big).
\end{equation}
Since $\frac{3}{4}<\sigma\leq 1$, it follows from taking Taylor series expansions that $a^{\ell}=1+O(\epsilon)$ and
\[
a-b = \sum_{n=1}^{\infty} \frac{(1-\sigma^n)(\epsilon/\ell)^n}{(n+1)!} \leq \sum_{n=1}^{\infty} \frac{n (1-\sigma)(\epsilon/\ell)^n}{(n+1)!} \ll  (1-\sigma) \frac{\epsilon}{\ell}.
\]
We apply these two Taylor expansions to \eqref{eqn:blue_harvest_1} and \eqref{eqn:blue_harvest_2} to obtain
\begin{equation}
\label{eqn:blue_harvest_3}
(\log x) ( F(-\log x) - F(-\sigma \log x) )=\Big(x - \frac{x^{\sigma}}{\sigma}\Big) (1 + O(\epsilon) ) + O\Big(\frac{x^{\sigma}}{\sigma} (1-\sigma)  \epsilon \Big)+O(x^{1/2}).
\end{equation}
Finally, we observe that since $\sigma^{-2} x^{\sigma} \leq x$ for $\sigma > 3/4$ and $x \geq 10$, we have that
\[
\frac{x^{\sigma}}{\sigma} (1-\sigma)=\sigma\Big(\frac{x^{\sigma}}{\sigma^2} - \frac{x^{\sigma}}{\sigma}\Big) \leq \sigma \Big( x - \frac{x^{\sigma}}{\sigma}\Big).
\]
We apply this observation to \eqref{eqn:blue_harvest_3} to obtain
\begin{equation}
	\label{eqn:blue_harvest_4}
	(\log x) ( F(-\log x) - F(-\sigma \log x) )=\Big(x-\frac{x^{\sigma}}{\sigma}\Big)(1+O(\epsilon))+O(x^{1/2}).
\end{equation}
The desired result follows by dividing both sides of \eqref{eqn:blue_harvest_4} by $\log x$.
\end{proof}
		
	Let $\ell \geq 2$ be an integer, $x \geq 3$, and $\epsilon \in (0,1/4)$. Define
	\begin{equation}
		\begin{aligned}
			\widetilde{\psi}_C(x; f) = \widetilde{\psi}_C(x, L/F; f) & = \sum_{\kn} \Lambda_F(\kn) \mathbf{1}_C(\kn) f\Big( \frac{\log \N\kn}{\log x}\Big), 
		\end{aligned}
		\label{def:psi_C_smooth}
	\end{equation}
	where $f = f(\, \cdot \, ; x, \ell, \epsilon)$ is given by \cref{lem:WeightChoice}. To understand $\psi_C$, it suffices to study the smooth variant $\widetilde{\psi}_C$. 

	\begin{lem} \label{lem:Unsmooth}
		Let $\ell \geq 2$ be an integer, $x \geq 3$, and $\epsilon \in (0,1/4)$. Then 
		\[
		\psi_C(x) \leq \widetilde{\psi}_C(x; f)    + O(n_F x^{1/2})  \leq \psi_C(x e^{\epsilon}). 
		\]
		Moreover, $\widetilde{\psi}_C(x;f) = \psi_C(x) + O(n_F x^{1/2} + \epsilon x)$.
	\end{lem}
	\begin{proof} By \cref{lem:WeightChoice}(i,ii) and definitions \eqref{def:psi_C} and \eqref{def:psi_C_smooth}, we observe that
	\[
	\sum_{\sqrt{x} \leq \N\kn \leq x} \Lambda_F(\kn) \mathbf{1}_C(\kn)  \leq \widetilde{\psi}_C(x;f) \leq \psi_C(x e^{\epsilon}). 
	\]
	The lemma now follows from \eqref{def:psi_C} and the trivial estimate
	\[
	\sum_{z \leq \N\kn \leq y} \Lambda_F(\kn) \mathbf{1}_C(\kn) \leq n_F \sum_{z \leq n \leq y} \Lambda_{\Q}(n) \ll n_F (y-z) \qquad \text{ for $2 \leq z \leq y$.}
	\]	
	\end{proof}

\subsection{Dedekind zeta functions and Hecke $L$-functions} Now, assume $L/K$ is an abelian extension of number fields. The Dedekind zeta function $\zeta_L(s)$ satisfies
	\begin{equation}
	\zeta_L(s) = \prod_{\chi} L(s,\chi,L/K),
	\label{eqn:Factorization}
	\end{equation}
	where $\chi$ runs over the irreducible 1-dimensional Artin characters of $\mathrm{Gal}(L/K)$. By class field theory, each Artin $L$-function $L(s,\chi,L/K)$ is equal to a Hecke $L$-function $L(s,\chi,K)$, where (abusing notation) $\chi$ is a certain primitive Hecke character of $K$.  For simplicity, write $L(s,\chi)$ in place of $L(s,\chi,L/K)$ or $L(s,\chi,K)$.  Let the integral $\kf_{\chi} \subseteq \cO_K$ denote the conductor associated to $\chi$.  For each $\chi$, there exist nonnegative integers $a(\chi)$ and $b(\chi)$ satisfying $a(\chi)+b(\chi)=n_K$ such that if we define
	\[
	\gamma(s,\chi) = \Big[ \pi^{-\frac{s}{2}} \Gamma\Big(\frac{s}{2}\Big) \Big]^{a(\chi)} \Big[ \pi^{-\frac{s+1}{2}} \Gamma\Big(\frac{s+1}{2}\Big) \Big]^{b(\chi)}
	\]
	and
	\[
	\delta(\chi) = \begin{cases}
		1 & \text{if $\chi$ is trivial} \\
		0 & \text{otherwise,}
	\end{cases}
	\]
	then $\xi(s,\chi) := [s(1-s)]^{\delta(\chi)} (D_K\mathrm{N}\mathfrak{f}_{\chi})^{s/2} \gamma(s,\chi) L(s,\chi)$ satisfies the functional equation
	\begin{equation}
		\xi(s,\chi) = \varepsilon(\chi) \xi(1-s,\bar{\chi}),
		\label{eqn:Functional_Equation}
	\end{equation}
	where $\varepsilon(\chi)$ is a complex number with unit modulus.  Furthermore, $\xi(s,\chi)$ is an entire function of order 1 which does not vanish at $s=0$.  Note $L(s,\chi)$ has a simple pole at $s=1$ if and only if $\chi$ is trivial. The non-trivial zeros $\rho$ of $L(s,\chi)$ (which are the zeros of $\xi(s,\chi)$) satisfy $0 < \Re(\rho) < 1$, and the trivial zeros $\omega$ of $L(s,\chi)$ (which offset the poles of $\gamma(s,\chi)$) are at the non-negative integers, each with order at most $n_K$.

	The Dedekind zeta function $\zeta_L(s)$ possesses the same qualities (by considering the case $K = L$ and $\chi$ trivial). Namely, its completed $L$-function is
	\begin{equation}
	\xi_L(s) = [s(1-s)] D_L^{s/2} \Big[ \pi^{-s/2} \Gamma(\tfrac{s}{2}) \Big]^{a_L}\Big[(2\pi)^{-s} \Gamma(\tfrac{s+1}{2}) \Big]^{b_L} \zeta_L(s)
	\label{eqn:Completed_Dedekind}
	\end{equation}
	for certain integers $a_L,b_L \geq 0$ satsifying $a_L + b_L = [L:\Q]$. The trivial zeros $\omega$ of $\zeta_L(s)$ are at the non-negative integers with orders
	\begin{equation}
		\mathop{\ord}_{s=\omega} \zeta_L(s) = \begin{cases}
 					a_L & \omega = -2,-4,\dots \\
 					b_L & \omega = -1,-3,\dots \\
 					a_L-1 & \omega = 0. 
 				\end{cases}
		\label{eqn:TrivialZeros}
	\end{equation}
	Moreover, the conductor-discriminant formula states that
	\begin{equation}
		\log D_L = \sum_{\chi} \log(D_K\mathrm{N}\mathfrak{f}_{\chi}).
		\label{eqn:conductor_discriminant}
	\end{equation}
	From \eqref{eqn:MaxConductor_Abelian} with $\cQ = \cQ(L/K)$, it follows that
	\begin{equation}
	\log D_L \leq [L:K] \log(D_K \cQ).
	\label{eqn:disc_MaxCondBound}
	\end{equation}
	From this we deduce a somewhat crude bound for $\log D_L$ in terms of $D_K, \cQ,$ and $n_K$. 
	\begin{lem} \label{lem:disc_MaxCondBound} If $L/K$ is abelian, then $\log D_L \ll (D_K \cQ n_K^{n_K})^2$. 
	\end{lem}
	\begin{proof}
	By class field theory, $L$ is contained in some ray class field $L'$ of $K$ whose Artin conductor has norm at most $\cQ$. From \cite[Lemma 1.16]{Weiss}, it follows that $[L:K] \leq [L':K] \leq D_K \cQ e^{O(n_K)}$. The result now follows from \eqref{eqn:disc_MaxCondBound}. \end{proof}
	
	We also record a few standard estimates for Hecke $L$-functions. 

\begin{lem}[Lemma 5.4 of \cite{LO}] \label{lem:NumberOfZeros} 
If $t \in \R$ and $\chi$ is a Hecke character of $K$, then
 \[
 \#\{ \rho = \beta + i\gamma\colon L(\rho, \chi) = 0,~0 < \beta < 1,~|\gamma-t| \leq 1 \} \ll \log(D_K \N\kf_{\chi}) + n_K \log(|t|+3),
 \]	
 where the zeros $\rho$ are counted with multiplicity. 
\end{lem}

\begin{lem}[Lemma 5.6 of \cite{LO}] \label{lem:CriticalStripGrowth} Let $\chi$ be a Hecke character of $K$. Then
\[
-\frac{L'}{L}(s,\chi) \ll  \log(D_K \N\kf_{\chi}) + n_K\log(|\Im(s)|+3) 
\]
uniformly for $\Re(s) = -1/2$.
\end{lem}

\section{The distribution of zeros}
\label{sec:Principles}
For Sections \ref{sec:Principles} and \ref{sec:MainTechnicalResult}, we will assume that the extension $L/K$ is abelian. For notational simplicity, define 
\begin{equation}
	Q = Q(L/K) := D_K \cQ n_K^{n_K},
	\label{def:Qmax}
\end{equation}
where $\cQ = \cQ(L/K)$ is given by \eqref{eqn:MaxConductor_Abelian}. 
Any sum $\sum_{\chi}$ or product $\prod_{\chi}$ is over the primitive Hecke characters $\chi$ associated with $L/K$ per the factorization in \eqref{eqn:Factorization}. Here we list three key results regarding the distribution of zeros of Hecke $L$-functions.

 \begin{thm}[Zero-free region] \label{thm:ZFR}
 	There exists $\Cl[abcon]{ZFR} > 0$ such that the Dedekind zeta function
 	\[
 	\zeta_L(s) = \prod_{\chi} L(s,\chi,L/K)
 	\]
 	has at most one zero in the region $\Re(s) > 1 - \Delta(|\Im(s)|+3)$, where the function $\Delta$ satisfies
 	\begin{equation}
		\Delta(t) \geq \frac{\Cr{ZFR}}{\log(Qt^{n_K})} \qquad \text{for $t \geq 3$.}
		\label{eqn:ZFR_constant}
	\end{equation} 	
	If such an {\it exceptional zero} $\beta_1$ exists then it is real, simple, and attached to the $L$-function of a real Hecke character $\chi_1$. 
 \end{thm}
\begin{proof}
	This is well-known; see, for example, \cite[Theorem 1.9]{Weiss}.  
\end{proof}
 
We also refer to the exceptional zero $\beta_1$ as a Landau-Siegel zero. Now, for $0 \leq \sigma \leq 1, T \geq 1$ and any Hecke character $\chi$, define
	\begin{equation}
	N(\sigma, T, \chi) = \#\{ \rho = \beta + i\gamma  : L(\rho,\chi) = 0, \sigma < \beta < 1, |\gamma| \leq T \}, 
	\label{def:NumberOfZeros}
	\end{equation}
	where the zeros $\rho$ are counted with multiplicity.

\begin{thm}[Log-free zero density estimate] \label{thm:LFZDE}
	There exists an integer $\Cl[abcon]{ZDE} \geq 1$ such that 
	\begin{equation}
	\sum_{\chi} N(\sigma,T,\chi) \ll B_1 (QT^{n_K})^{\Cr{ZDE}(1-\sigma)}
	\label{eqn:ZDE}
	\end{equation}
	uniformly for any $0 < \sigma < 1$ and $T \geq 1$, where
	\begin{equation}
	B_1 = B_1(T) = \min\{1,(1-\beta_1) \log(QT^{n_K})\}.
	\label{def:B_1}
	\end{equation}
\end{thm}
\begin{proof} Let $\epsilon_0>0$ be a sufficiently small absolute and effective constant.  It follows from \cite[Theorem 3.2]{TZ1} or its variant \cite[Theorem 4.5]{TZ2} that if $1-\epsilon_0<\sigma<1$ and $T\geq 1$, then
	\[
	\sum_{\chi} N(\sigma,T,\chi) \ll (QT^{n_K})^{\Cr{ZDE}(1-\sigma)}
	\]
	regardless of whether $\beta_1$ exists.  Weiss \cite[Theorem 4.3]{Weiss} proved that if $\beta_1$ exists, then for $1-\epsilon_0<\sigma<1$ and $T\geq 1$,
	\[
	\sum_{\chi} N(\sigma,T,\chi) \ll (1-\beta_1) \log(QT^{n_K})(QT^{n_K})^{\Cr{ZDE}(1-\sigma)}.
	\]
	Thus for $T\geq 1$, \eqref{eqn:ZDE} holds with $B_1$ given by \eqref{def:B_1} in the range $1-\epsilon_0<\sigma<1$.   By enlarging $c_5$ if necessary and using Stark's bound from \cref{thm:Stark}, one can extend \eqref{eqn:ZDE} to the remaining interval $0 < \sigma < 1-\epsilon_0$ by employing the trivial bound that follows from \cref{lem:NumberOfZeros}. 
\end{proof}

\cref{thm:ZFR,thm:LFZDE} comprise the three principles used to prove Linnik's theorem on the least prime in an arithmetic progression: a zero-free region, a log-free zero density estimate, and a quantitative form of the zero repulsion phenomenon.  \cref{thm:LFZDE} combines the second and third principles by following the ideas of Bombieri \cite{Bombieri2}, and this is crucial to our arguments for certain choices of Galois extensions (see \cref{sec:appendix}).

 Additionally, one needs a more refined effective lower bound for the size of $1-\beta_1$ than what \eqref{eqn:Stark} provides.  Such a lower bound follows from \cite[Theorem 1', p. 148]{Stark-1974}.

	\begin{thm}[Stark's bound] \label{thm:Stark}
		Let $\beta_1 = 1-\lambda_1/\log Q$ be a real zero of a real Hecke character $\chi$ of the abelian extension $L/K$. Then $\lambda_1 \gg Q^{-2}$.
	\end{thm}
	\begin{proof}
		This follows readily from \eqref{eqn:Stark} for $1-\beta$ when $\beta$ is the real zero of a Dedekind zeta function. If $\chi$ is trivial then consider the Dedekind zeta function $\zeta_K(s)$. If $\chi$ is quadratic then consider the Dedekind zeta function $\zeta_K(s) L(s,\chi,L/K)$ corresponding to the quadratic extension of $K$ defined by $\chi$. 
	\end{proof}

As we shall see, these three theorems yield a unified Chebotarev density theorem which produces an asymptotic count for primes even in the presence of a Landau-Siegel zero. 

\section{Weighted counts of primes in abelian extensions}
\label{sec:MainTechnicalResult}

\subsection{Main technical result} The proof of \cref{thm:main_theorem} rests on the analysis on the weighted prime counting function $\widetilde{\psi}_C(x; f) = \widetilde{\psi}_C(x, L/K; f)$ given by \eqref{def:psi_C_smooth}, where  $f$ is given by \cref{lem:WeightChoice} and $L/K$ is abelian. The goal of this section is to prove the following proposition. 

\begin{prop} \label{prop:WeightedPrimeSum}
Assume $L/K$ is abelian with Galois group $G$. Let $C \subseteq G$ be a conjugacy class of $G$. Let $f = f(\, \cdot \, ; x, \ell, \epsilon)$ be defined as in \cref{lem:WeightChoice} with 
\begin{equation}
\epsilon = 8\ell x^{-1/8\ell}, \qquad \ell = 4 \Cr{ZDE} n_K.
\label{eqn:ParameterChoices}
\end{equation}	
If $2 \leq Q  \leq x^{1/(36\Cr{ZDE})}$ and $\epsilon < 1/4$, then
\begin{equation}
\frac{|G|}{|C|} \widetilde{\psi}_C(x;f) = \Big( x - \chi_1(C) \frac{x^{\beta_1}}{\beta_1}  \Big) \Big( 1 + O\big( e^{-\frac{\Cr{ZFR}}{2} \frac{\log x}{\log Q}} + e^{-\sqrt{\Cr{ZFR} (\log x)/4n_K}} \big) \Big). 
\label{eqn:Asymptotic_Rewrite}
\end{equation}
\end{prop}
\begin{remark}
The constants $\Cr{ZFR}$ and $\Cr{ZDE}$ are defined in \cref{thm:ZFR,thm:LFZDE} respectively. 	 
\end{remark}

While $f$ and its parameters are chosen in \cref{prop:WeightedPrimeSum}, we will assume throughout this section that $\epsilon \in (0,1/4)$ and $\ell \geq 2$ are arbitrary, unless otherwise specified. The arguments leading to \cref{prop:WeightedPrimeSum} are divided into natural steps: shifting a contour, estimating the arising zeros with the log-free zero density estimate, and optimizing the error term with a classical zero-free region.

\subsection{Shifting the contour}

\begin{lem} \label{lem:ContourShift}
	 If $x \geq 3$, then $\dfrac{|G|}{|C|}\dfrac{\widetilde{\psi}_C(x;f)}{\log x}$ equals
	\[
		F(-\log x) -  \chi_1(C)F(-\beta_1 \log x) - \sum_{\chi} \bar{\chi}(C)\sumS_{ \rho_{\chi} } F(-\rho_{\chi} \log x)  + O\Big(  \frac{(2\ell/\epsilon)^{\ell} \log D_L}{x^{1/4} \log x} + \frac{n_L}{\log x} \Big),
	\]
	where the sum $\sum^{\star}$ is over all non-trivial zeros $\rho_{\chi} \neq \beta_1$ of $L(s,\chi)$, counted with multiplicity. Here the term $F(-\beta_1 \log x)$ may be omitted if the exceptional zero $\beta_1$ does not exist. 
\end{lem}
\begin{proof} By \eqref{def:psi_C_integral}, \eqref{def:psi_C_smooth}, \cref{lem:WeightChoice} and a standard Mellin inversion calculation,
\begin{equation}
\frac{|G|}{|C|}\widetilde{\psi}_C(x;f) =   \sum_{\chi} \bar{\chi}(C) I_{\chi}, \quad \text{where } I_{\chi} = \frac{\log x}{2\pi i} \int_{2-i\infty}^{2+i\infty} -\frac{L'}{L}(s,\chi) F(-s \log x) ds.
	\label{eqn:MellinInversion}
	\end{equation}
	For each Hecke character $\chi$, shift the contour $I_{\chi}$ to the line $\Re(s) = -1/2$.  Note $F$ is entire by \cref{lem:WeightChoice}(iii), so we need only consider the zeros and poles of $L(s,\chi)$.  We pick up the simple pole at $s=1$ of $L(s,\chi)$ when $\chi$ is trivial and the trivial zero at $s=0$ of $L(s,\chi)$ of order at most $n_K$. Moreover, we also pick up all of the non-trivial zeros $\rho_\chi$ of $L(s,\chi)$. For the remaining contour along $\Re(s) = -1/2$, we apply \cref{lem:CriticalStripGrowth}, Minkowski's estimate $n_K \ll \log D_K$, and \cref{lem:WeightChoice}(vi) to deduce that
	\[
	-\frac{\log x}{2\pi i} \int_{-1/2-i\infty}^{-1/2+i\infty} \frac{L'}{L}(s,\chi,L/K) F(-s \log x)ds \ll \frac{(2\ell/\epsilon)^{\ell} \log(D_K \N\kf_{\chi})}{x^{1/4}}.
	\]
	Combining all of these observations yields
	\begin{equation}
		(\log x)^{-1} I_{\chi}  = \delta(\chi) F(-\log x) - \sideset{}{^\star}\sum_{\rho_\chi} F(-\rho_{\chi} \log x) + O\Big( F(0) n_K +  \frac{(2\ell/\epsilon)^{\ell} \log(D_K \N\kf_{\chi})}{x^{1/4} \log x} \Big).
		\label{eqn:Contour_If}
	\end{equation} 
	Here, $\rho_{\chi}$ runs over all non-trivial zeros of $L(s,\chi)$, including $\beta_1$ if it exists. Substituting \eqref{eqn:Contour_If} into \eqref{eqn:MellinInversion} and dividing through by $\log x$, we obtain the desired result but with an error term of 
	\[
	O\Big( \frac{|F(0)| n_K}{\log x} \sum_{\chi} |\bar{\chi}(C)|   + \frac{(2\ell/\epsilon)^{\ell}  }{x^{1/4} \log x} \sum_{\chi}  |\bar{\chi}(C)|  \log(D_K \N\kf_{\chi})    \Big). 
	\]
	As $L/K$ is abelian, the characters $\chi$ are 1-dimensional so $|\bar{\chi}(C)| = 1$. Thus, applying the conductor-discriminant formula \eqref{eqn:conductor_discriminant}, the observation $n_K \sum_{\chi} 1 = [L:K] n_K = n_L$, and \cref{lem:WeightChoice}(iv), we obtain the desired error term. 
\end{proof}

\subsection{Estimating the zeros}

Now we estimate the sum over non-trivial zeros $\rho$ in \cref{lem:ContourShift}, beginning with those $\rho$ of small modulus. 
\begin{lem} \label{lem:SumOverSmallZeros} If $x \geq 3$, then $\displaystyle\sum_{\chi} \sum_{\substack{\rho_{\chi} \\ |\rho_{\chi}| \leq 1/4}} |F(-\rho_{\chi} \log x)| \ll x^{1/4} \log D_L$.
\end{lem}
\begin{proof}
	From \cref{lem:WeightChoice}(iv) and \cref{lem:NumberOfZeros}, 
	\begin{equation*}
	\sum_{\chi} \sideset{}{^\star}\sum_{\substack{\rho_{\chi} \\ |\rho_{\chi}| \leq 1/4}} |F(-\rho_{\chi} \log x)| 	
		 \ll \sum_{\chi} \sideset{}{^\star}\sum_{\substack{\rho_{\chi} \\ |\rho_{\chi}| \leq 1/4}} x^{1/4} 
		 \ll x^{1/4} \sum_{\chi} (\log (D_K \N\kf_{\chi}) + n_K ).
	\end{equation*}
	The result now follows from Minkowski's estimate $n_K \ll \log D_K$ and \eqref{eqn:conductor_discriminant}. 
\end{proof}

Next, we use the log-free zero density estimate to analyze the remaining contribution. 

\begin{lem} \label{lem:SumOverZeros}
Keep the assumptions and notation of \cref{lem:ContourShift}. Select $\epsilon$ and $\ell$ as in \eqref{eqn:ParameterChoices} and assume $\epsilon < 1/4$. For $2 \leq Q \leq x^{1/(8\Cr{ZDE})}$,
 \begin{equation}
\log x \sum_{\chi}   \sumS_{\substack{\rho_{\chi} \\ |\rho_{\chi}| \geq 1/4}} |F(-\rho_{\chi} \log x)|  \ll \nu_1  x e^{-\eta(x)/2},
 \label{eqn:SumOverZeros}
 \end{equation}
 where
  \begin{equation}
 	\nu_1 = \begin{cases}
 				(1-\beta_1)\log Q & \text{if $\beta_1$ exists}, \\
 				1 & \text{otherwise,}
 			\end{cases}
 	\label{def:nu}
 \end{equation}  
 and $\eta$ is given by
	 \begin{equation}
	 \begin{aligned}
 	\eta(x) & = \inf_{t \geq 3}\big[ \Delta(t) \log x +  \log t  \big]. \\
 	\end{aligned}
 	\label{def:eta_function}
 	\end{equation}
\end{lem}

\begin{proof} We dyadically estimate the zeros. For $j \geq 1$, set $T_0 = 0$ and $T_j = 2^{j-1}$ for $j \geq 1$. Consider the sum
	\begin{equation}
	Z_j := \frac{\log x}{x} \sum_{\chi} \sum_{\substack{ \rho_{\chi} = \beta_{\chi} + i\gamma_{\chi} \\ T_{j-1} \leq |\gamma_{\chi}| \leq T_j  \\ |\rho_{\chi}| \geq 1/4 }} |F(-\rho_{\chi} \log x)|
	\label{eqn:Dyadic_ContributionZeros}
	\end{equation}	
	for $j \geq 1$.  First, we estimate the contribution of each zero $\rho = \rho_{\chi}$ appearing in $Z_j$. Let $\rho = \beta+i\gamma$ satisfy $T_{j-1} \leq |\gamma| \leq T_j$ and $|\rho| \geq 1/4$, so $|\rho| \geq \max\{T_{j-1},\tfrac{1}{4} \} \geq T_j/4$ and $|\rho| \gg |\gamma|+3$.  Thus, \cref{lem:WeightChoice}(iv) with $\alpha = \ell(1-\beta)$ and our choice of $\epsilon$ imply that
	\begin{align*}
	\frac{\log x}{x} |F(-\rho \log x)|  \ll \frac{x^{\beta-1}}{|\rho|}   \Big( \frac{2\ell}{\epsilon |\rho|} \Big)^{\ell(1-\beta)} \ll T_j^{-1/2} (|\gamma|+3)^{-1/2} \cdot  x^{-(1-\beta)/2}  \cdot \big( x^{3/8} T_j^{\ell} \big)^{-(1-\beta)}.
 	\end{align*}
	Since $Q \leq x^{1/(8\Cr{ZDE})}$ and $\ell = 4\Cr{ZDE} n_K$, it follows that
	\begin{equation}
	\frac{\log x}{x} |F(-\rho \log x)| \ll T_j^{-1/2} \cdot (|\gamma|+3)^{-1/2} x^{-(1-\beta)/2} (QT_j^{n_K})^{-2\Cr{ZDE}(1-\beta)}. 
	\label{eqn:ContributionZero}
	\end{equation}
	From \cref{thm:ZFR} and \eqref{def:eta_function}, we deduce
	\[
	(|\gamma|+3)^{-1/2} x^{-(1-\beta)/2} \leq (|\gamma|+3)^{-1/2} x^{-\Delta(|\gamma|+3)/2} \leq e^{-\eta(x)/2}. 
	\]
	Note the righthand side is uniform over all non-trivial zeros $\rho$ appearing in \eqref{eqn:SumOverZeros}. Combining \eqref{eqn:ContributionZero} and the above inequality with \eqref{eqn:Dyadic_ContributionZeros}, we deduce that 
	\[
	Z_j \ll e^{-\eta(x)/2} T_j^{-1/2}  \sum_{\chi} \sum_{\substack{ \rho_{\chi} = \beta_{\chi} + i\gamma_{\chi} \\ T_{j-1} \leq |\gamma_{\chi}| \leq T_j }} (QT_j^{n_K})^{-2\Cr{ZDE}(1-\beta)}.
	\]
	Defining $N(\sigma,T) = \sum_{\chi} N(\sigma,T,\chi)$, we use partial summation and \cref{thm:LFZDE} to see that
	\begin{equation*}
	\begin{aligned}
		e^{\eta(x)/2}T_j^{1/2} Z_j & \ll  \int_0^1 (QT_j^{n_K})^{-2\Cr{ZDE}\alpha} dN(1-\alpha,T_j) \\		
			& \ll  \Big[ (QT_j^{n_K})^{-2\Cr{ZDE}} N(0,T_j) + \log(QT_j^{n_K}) \int_0^1 (QT_j^{n_K})^{-2\Cr{ZDE}\alpha} N(1-\alpha,T_j) d\alpha \Big] \\		
			& \ll B_1(T_j)  \Big[ (QT_j^{n_K})^{-\Cr{ZDE}}  + \log(QT_j^{n_K}) \int_0^1 (QT_j^{n_K})^{-\Cr{ZDE}\alpha}   d\alpha \Big]  \ll B_1(T_j).  
	\end{aligned}
	\end{equation*}
	If a Landau-Siegel zero does not exist then $B_1(T_j) = 1 = \nu_1$. Otherwise, if a Landau-Siegel zero exists then one can verify by \eqref{def:B_1} and a direct calculation that
	\[
	B_1(T_j)T_j^{-1/4} \leq (1-\beta_1) \cdot \sup_{t \geq 1}\big[ \log(Qt^{n_K})t^{-1/4}\big] \ll (1-\beta_1) \log Q = \nu_1.
	\]
	The supremum occurs at $t \ll 1$ since $n_K \leq \log Q$. Therefore, 
	\[
	\sum_{j \geq 1} Z_j \ll e^{-\eta(x)/2} \sum_{j \geq 1} \frac{B_1(T_j)}{T_j^{1/4}} \cdot \frac{1}{T_j^{1/4}} \ll \nu_1 e^{-\eta(x)/2} \sum_{j \geq 1} 2^{-j/4} \ll \nu_1 e^{-\eta(x)/2},
	\]
	which yields the lemma by definition \eqref{eqn:Dyadic_ContributionZeros}. 
\end{proof}

\subsection{Error term with a classical zero-free region}
The quality of the error term in \cref{lem:SumOverZeros}, and hence in \cref{prop:WeightedPrimeSum}, is reduced to computing $\eta(x)$.  This is a single-variable optimization problem.

\begin{lem} \label{lem:ErrorOptimization_Classical}
Let $\eta$ be defined by \eqref{def:eta_function}. If $x \geq 2$ then $e^{-\eta	(x)} \leq e^{- \Cr{ZFR} \frac{\log x}{\log Q} } + e^{- \sqrt{\Cr{ZFR} (\log x)/n_K}}$.
\end{lem}
\begin{proof} It follows from \cref{thm:ZFR}, \eqref{def:eta_function}, and a change of variables $t = e^u$ that
	\[
	\eta(x) \geq \inf_{u \geq 0} \phi_x(u) \qquad \text{where } \phi_x(u) = \frac{\Cr{ZFR}  \log x}{\log Q + n_K u } + u. 
	\]
	Note that $\phi_x(u) \to \infty$ as $u \to \infty$. 	By standard calculus arguments, 
	one can verify that
		\begin{equation}
	 \eta(x) \geq \begin{cases}
 			 \frac{\Cr{ZFR}\log x}{\log Q}  & \text{if } 2 \leq x \leq  \exp(\frac{(\log Q)^2}{\Cr{ZFR} n_K}), \\[2mm]
 			\sqrt{\frac{\Cr{ZFR} \log x}{n_K}} & \text{ if }  x \geq  \exp(\frac{(\log Q)^2}{\Cr{ZFR} n_K}). 
		 \end{cases}
		 \label{def:UV_Phi}
	\end{equation}
	This proves the lemma. 
	\end{proof}

\subsection{Proof of \cref{prop:WeightedPrimeSum}} \label{proof:WeightedPrimeSum}
Choose $\epsilon$ and $\ell$ as in \eqref{eqn:ParameterChoices} and continue to assume $\epsilon < 1/4$. By \cref{lem:ContourShift,lem:SumOverSmallZeros,lem:SumOverZeros}, it follows for $2 \leq Q \leq x^{1/(36\Cr{ZDE})}$ that
\[
\frac{|G|}{|C|}\widetilde{\psi}_C(x;f)    =  (\log x) \big[ F(-\log x) - \chi_1(C) F(-\beta_1 \log x) \big] + O\big( \nu_1 x e^{-\eta(x)/2} + \mathcal{E}(x) \big),
\]
where $\mathcal{E}(x) =  x^{-1/4}  (2\ell/\epsilon)^{\ell} \log D_L  + n_L  + x^{1/4} (\log x)(\log D_L)$.  From \eqref{eqn:ParameterChoices} and Minkowski's estimate $n_L \ll \log D_L$, we see that $\mathcal{E}(x)  \ll x^{1/4} (\log D_L) (\log x)$.  From \cref{lem:disc_MaxCondBound}, $\log D_L \ll Q^{2} \ll x^{1/10}$ since $x \geq Q^{36 \Cr{ZDE}}$ and $\Cr{ZDE} \geq 1$. Hence, $\mathcal{E}(x) \ll x^{1/2}$. Using \cref{lem:WeightChoice}(v), \eqref{eqn:ParameterChoices}, and noting $\beta_1 > 1/2$, we deduce that
	\begin{equation}
\frac{|G|}{|C|} \widetilde{\psi}_C(x;f) = \Big( x - \chi_1(C) \frac{x^{\beta_1}}{\beta_1}  \Big) \big( 1 + O(n_K x^{-\frac{1}{32\Cr{ZDE}n_K}}) \big) + O( \nu_1 x e^{-\eta(x)/2} + x^{1/2}) 
\label{eqn:proof_WeightedPrimeSum}
\end{equation}
for $2 \leq Q \leq x^{1/36c_5}$. Now, we claim that
\begin{equation}
x - \chi_1(C) \frac{x^{\beta_1}}{\beta_1}  \gg \nu_1 x \gg x^{3/4}. 
\label{eqn:MainTerm_LB}
\end{equation}
If $\beta_1$ does not exist, then $\nu_1 = 1$ and \eqref{eqn:MainTerm_LB} is immediate. 
If $\beta_1$ exists and $(1-\beta_1) \log x < 1$, then since  $x \geq Q^{36\Cr{ZDE}}$ and $e^{-t} \geq 1-t$ for $0 < t < 1$, we have
\[
x - \chi_1(C) \frac{x^{\beta_1}}{\beta_1} \geq  x\Big(1- \frac{x^{-(1-\beta_1)}}{\beta_1} \Big) \geq (1-\beta_1) x \log(x/e) \gg (1-\beta_1) x \log Q = \nu_1 x. 
\]
Otherwise, $\beta_1$ exists and $(1-\beta_1) \log x \geq 1$ so $\beta_1 > 1/2$ implies that
\[
x - \chi_1(C) \frac{x^{\beta_1}}{\beta_1} \geq 
 x\Big(1- \frac{x^{-(1-\beta_1)}}{\beta_1} \Big)
 \geq x(1- 2e^{-1})\gg x \gg \nu_1 x,
\]
Thus, the claim \eqref{eqn:MainTerm_LB} follows upon noting that $\nu_1 \gg  Q^{-2} \gg x^{-1/4}$ by Stark's bound (\cref{thm:Stark}) and the condition $x \geq Q^{36 \Cr{ZDE}}$. Combining \eqref{eqn:MainTerm_LB} with \eqref{eqn:proof_WeightedPrimeSum}, it follows that
\begin{equation}
\frac{|G|}{|C|} \widetilde{\psi}_C(x;f) = \Big( x - \chi_1(C) \frac{x^{\beta_1}}{\beta_1}  \Big) \big( 1 + O(e^{-\eta(x)/2} + n_K x^{-\frac{1}{32\Cr{ZDE}n_K}}) \big). 
\label{eqn:WeightPrimeSum_Penultimate}
\end{equation}
Finally, we apply \cref{lem:ErrorOptimization_Classical} and note $n_K x^{-1/(32 \Cr{ZDE} n_K)} \ll x^{-1/(300 \Cr{ZDE} n_K)} \ll e^{-\sqrt{\Cr{ZFR}(\log x)/(4 n_K)}}$ for $x \geq Q^{36 \Cr{ZDE}}$.   This completes the proof of \cref{prop:WeightedPrimeSum}. 
\hfill
\qed

\section{Proof of \cref{thm:CDT_1,thm:main_theorem}}
\label{sec:ProofofMainTheorem}

\subsection{Abelian extensions}

First, we prove \cref{thm:main_theorem} in the case of abelian extensions. 
\begin{thm} \label{thm:NaturalPrimeSum}
	 Assume $L/K$ is abelian with Galois group $G$. Let $C \subseteq G$ be a conjugacy class. Define $Q$ by \eqref{def:Qmax}.  For $2 \leq Q \leq x^{1/\Cr{1}}$, 
	\begin{equation}
		\label{eqn:NaturalPrimeSum_Final}
		\begin{aligned}
	 \pi_{C}(x,L/K) & = \frac{|C|}{|G|} \Big( \Li(x) - \chi_1(C) \Li(x^{\beta_1}) \Big) \Big(1 + O\Big( e^{- \frac{\Cr{ZFR}}{4}\frac{\log x}{\log Q}}  +  e^{- \sqrt{\Cr{ZFR} (\log x)/8n_K}} \Big) \Big).
		\end{aligned}
	\end{equation}
	Here $\beta_1$ is a putative exceptional zero with associated real Hecke character $\chi_1$ of $L/K$.  
\end{thm}
\begin{proof}
Write $g(x) = x - \chi_1(C) \frac{x^{\beta_1}}{\beta_1}$. Select $\epsilon$ as in \eqref{eqn:ParameterChoices}. Note the assumption $2 \leq Q \leq x^{1/\Cr{1}}$  guarantees $\epsilon < 1/4$ provided $\Cr{1}$ is sufficiently large. From \cref{prop:WeightedPrimeSum} and \cref{lem:Unsmooth}, it follows that
\begin{equation}
\psi_C(x) \leq  \frac{|C|}{|G|}  g(x) \big( 1 + O(  e^{-\frac{\Cr{ZFR}}{2} \frac{\log x}{\log Q}} + e^{-\sqrt{\Cr{ZFR} (\log x)/4n_K}} ) \big) \qquad \text{for $x \geq Q^{36\Cr{ZDE}}$}.
\label{eqn:psi_C_upperbound}
\end{equation}
On the other hand, writing $y = x e^{\epsilon}$, \cref{prop:WeightedPrimeSum,lem:Unsmooth} also imply
\[
\psi_C(y) \geq \frac{|C|}{|G|} g(y e^{-\epsilon})\big(1 + O(  e^{-\frac{\Cr{ZFR}}{2} \frac{\log y}{\log Q}} + e^{-\sqrt{\Cr{ZFR} (\log y)/4n_K}} ) \big)
\]
for $y \geq 2Q^{36\Cr{ZDE}}$. By \eqref{eqn:MainTerm_LB} and elementary arguments, 
\[
\big| g(y e^{-\epsilon}) - g(y) e^{-\epsilon} \big| \leq \frac{y^{\beta_1}}{\beta_1} (e^{-\epsilon \beta_1} - e^{-\epsilon}) \ll y \epsilon(1-\beta_1) \ll \epsilon g(y). 
\]
In particular, $g(y e^{-\epsilon}) = g(y)(1+O(\epsilon))$. From our choice of $\epsilon$ in \eqref{eqn:ParameterChoices} and the condition $y \geq 2 Q^{36 \Cr{ZDE}}$, one can see that $\epsilon \ll n_K y^{-1/32 \Cr{ZDE} n_K} \ll y^{-1/300 \Cr{ZDE}n_K} \ll e^{-\sqrt{\Cr{ZFR} (\log y)/4n_K}}$ so
\[
\psi_C(y) \geq \frac{|C|}{|G|} g(y) \big(1 +O(  e^{-\frac{\Cr{ZFR}}{2} \frac{\log y}{\log Q}} + e^{-\sqrt{\Cr{ZFR} (\log y)/4n_K}} ) \big) \qquad \text{for $y \geq 2 Q^{36 \Cr{ZDE}}$.}
\]
Comparing the above with \eqref{eqn:psi_C_upperbound}, we conclude that
\[
\psi_C(x) = \frac{|C|}{|G|} g(x) \big(1 +O(  e^{-\frac{\Cr{ZFR}}{2} \frac{\log x}{\log Q}} + e^{-\sqrt{\Cr{ZFR} (\log x)/4n_K}} )
\]
for $x \geq Q^{40\Cr{ZDE}}$.  By partial summation (\cref{lem:Pi_to_Psi}) and the observation that, for $1/2 < \sigma \leq 1$, 
\begin{equation}
\frac{x^{\sigma}}{\sigma \log x} + \int_{\sqrt{x}}^x \frac{t^{\sigma-1}}{\sigma (\log t)^2} dt  = \int_{x^{\sigma/2}}^{x^{\sigma}} \frac{1}{\log t} dt = \Li(x^{\sigma}) + O\Big( \frac{x^{1/2}}{\log x} \Big), 
\label{eqn:Li_identity}
\end{equation}
it follows for $x \geq Q^{40 \Cr{ZDE}}$  that
\begin{equation*}
\frac{|G|}{|C|} \pi_C(x) = \Big(\Li(x) - \chi_1(C) \Li(x^{\beta_1}) \Big) \Big(1 +O(  e^{-\frac{\Cr{ZFR}}{4} \frac{\log x}{\log Q}} + e^{-\sqrt{\Cr{ZFR} (\log x)/8n_K}} ) \Big) + \mathcal{E}_0(x),
\label{eqn:Pi_penultimate}
\end{equation*}
where $\mathcal{E}_0(x) =   \log D_L + n_K x^{1/2}/\log x$.  By \cref{lem:disc_MaxCondBound} and the observation that $n_K \ll \log x$, one can verify that $\mathcal{E}_0(x) \ll x^{1/2}$ for $x \geq Q^{40 \Cr{ZDE}}$. Hence, by \eqref{eqn:MainTerm_LB}, $\mathcal{E}_0(x)$ can be absorbed into the error term of \eqref{eqn:Pi_penultimate}. As $\Cr{1}$ is sufficiently large, this completes the proof of \cref{thm:NaturalPrimeSum}.
\end{proof}

\subsection{Proof of \cref{thm:main_theorem}} Now we finish the proof of \cref{thm:main_theorem} for any Galois extension $L/F$ with any Galois group $G$. Using well-known arguments from class field theory, we reduce to the case of abelian extensions. 

	\begin{lem}[Murty-Murty-Saradha] \label{lem:ReductionToAbelian}
	Let $L/F$ be a Galois extension of number fields with Galois group $G$, and let $C\subseteq G$ be a conjugacy class.  Let $H$ be a subgroup of $G$ such that $C\cap H$ is nonempty, and let $K$ be the fixed field of $L$ by $H$.  Let $g\in C\cap H$, and let $C_H(g)$ denote the conjugacy class of $H$ which contains $g$.  If $x\geq 2$, then
	\[
	\Big|\pi_C(x,L/F)-\frac{|C|}{|G|}\frac{|H|}{|C_H|}\pi_{C_H}(x,L/K)\Big|\leq\frac{|C|}{|G|}\Big(n_L x^{1/2}+\frac{2}{\log 2}\log D_L\Big).
	\]
	\end{lem}
	\begin{proof}
		This is carried out during the proof of \cite[Proposition 3.9]{MMS}.
	\end{proof}
	Now, we apply \cref{lem:ReductionToAbelian} and subsequently \cref{thm:NaturalPrimeSum} to $\pi_{C_H}(x,L/K)$ of the abelian extension $L/K$. Consequently, for $2 \leq Q \leq x^{1/\Cr{1}}$, 
	\begin{equation}
	\begin{aligned}
	\frac{|G|}{|C|} \pi_C(x,L/F) 
	& =   \Big( \Li(x) - \chi_1(C) \Li(x^{\beta_1}) \Big) \Big(1 + O\Big( e^{- \frac{\Cr{ZFR}}{4}\frac{\log x}{\log Q}}  +  e^{- \sqrt{\Cr{ZFR} (\log x)/8n_K}} \Big) \Big) \\
	& \qquad \qquad + O( n_L x^{1/2} + \log D_L), 
	\end{aligned}
	\label{eqn:main_theorem_Penultimate}
	\end{equation}
	where $Q = Q(L/K)$ is defined by \eqref{def:Qmax}. Since we may assume $\Cr{1} \geq 20$, it follows from \cref{lem:disc_MaxCondBound} and Minkowski's estimate $n_L \ll \log D_L$ that $n_L x^{1/2} + \log D_L \ll x^{5/8}$ for $x \geq Q^{\Cr{1}}$. From \eqref{eqn:MainTerm_LB}, this estimate may be absorbed into the first error term of \eqref{eqn:main_theorem_Penultimate} since $x^{5/8-3/4} = x^{-1/8} \ll e^{-\sqrt{\Cr{ZFR}(\log x)/8n_K}}$. This completes the proof of \cref{thm:main_theorem}. \hfill \qed

\begin{proof}[\cref{thm:main_theorem} implies \cref{thm:CDT_1}]
Fix $g\in C$, let $H$ in \cref{thm:main_theorem} be the cyclic group generated by $g$, and let $K$ be the fixed field of $H$.  Clearly $n_K\leq n_L$, and the centered equation immediately below \cite[Equation 1-7]{TZ1} states $D_L^{1/|H|}\leq D_K\mathcal{Q}\leq D_L^{1/\varphi(|H|)}$.  \cref{thm:CDT_1} now follows.
\end{proof}

\section{Reduced composition of beta-sieves}
\label{sec:Sieve}

Before proceeding to the proof of \cref{thm:application}, we require some sieve machinery that follows from standard results. The setup and discussion here closely follow \cite[Sections 5.9 and 6.3--6.5]{FI}. Let $\Lambda'$ and $\Lambda''$ be beta sieve weights with the same sifting level $z$ and same level of distribution $R$. That is, $\lambda_{d}'$ and $\lambda_{d}''$   satisfy 
\[
\lambda'_{1}=\lambda''_{1}=1, \qquad |\lambda_{d}'| \leq 1, \qquad |\lambda_{d}''| \leq 1,
\]
 and are supported on squarefree numbers $d < R$ consisting of prime factors $\leq z$. Let
\[
s = \frac{\log R}{\log z}
\] 
be the sifting variable for both sieves. 
Let $g'$ and $g''$ be multiplicative functions satisfying 
\begin{equation}
0 \leq g'(p)  < 1, \qquad 0 \leq g''(p) < 1, \qquad  g'(p) + g''(p) < 1 \qquad \text{for all primes $p$}.
\label{eqn:size_g}
\end{equation}
Assume there exists $K > 1$ and $\kappa > 0$ such that
\begin{equation}
	\begin{aligned}
	\prod_{w \leq p < z} \Big(1 - \frac{g'(p)}{1-g'(p)-g''(p)} \Big)^{-1} & \leq K  \Big( \frac{\log z}{\log w}\Big)^{\kappa} \qquad \text{and} \\\prod_{w \leq p < z} \Big(1 - \frac{g''(p)}{1-g'(p)-g''(p)} \Big)^{-1} & \leq K  \Big( \frac{\log z}{\log w}\Big)^{\kappa} \qquad \text{for all $2 \leq w \leq z$}. 
	\end{aligned}
	\label{eqn:sieve_dimension}
\end{equation}
The goal of this section is to estimate the reduced composition given by
\begin{equation}
	G := \dsum_{\gcd(d_1,d_2)=1} \lambda'_{d_1}  \,\lambda_{d_2}''\,g'(d_1) \, g''(d_2).
	\label{def:reduced_composition}
\end{equation}
 This expression can arise as the main term when two different sieves are applied to two different sequences that are linearly independent.   Keeping this setup, the remainder of this section will be dedicated to the proof of the following theorem.
\begin{thm} \label{thm:CompositionBetaSieves}
	 Assume $s > 9\kappa + 1 + 10 \log K$, \eqref{eqn:size_g} holds, and \eqref{eqn:sieve_dimension} holds. If $\Lambda'$ and $\Lambda''$ are upper bound beta sieves, then
	\[
	\dsum_{\gcd(d_1,d_2)=1} \lambda'_{d_1}  \,\lambda_{d_2}''\,g'(d_1) \, g''(d_2) \leq \prod_p (1-g'(p) - g''(p)) \big\{ 1 +  e^{9\kappa - s} K^{10}\}^2.
	\]
	If $\Lambda'$ is a lower bound beta sieve and $\Lambda''$ is an upper bound beta sieve, then
	\[
	\dsum_{\gcd(d_1,d_2)=1} \lambda'_{d_1}  \,\lambda_{d_2}''\,g'(d_1) \, g''(d_2) \geq \prod_p (1-g'(p)-g''(p)) \big\{ 1 -   e^{9\kappa - s} K^{10}\}.
	\]
	
\end{thm}

Assume $\lambda'$ is a lower bound beta sieve and $\lambda''$ is an upper bound beta sieve. The other case is entirely analogous. Thus, if $\theta' =1 \ast \lambda'$ and $\theta'' = 1 \ast \lambda''$ then
\begin{equation}
\theta'_1 = \theta''_1 = 1 \qquad \text{and} \qquad \theta'_n \leq 0 \leq \theta''_n \quad \text{for $n \geq 2$}. 
\label{eqn:lowerbound_sieve}
\end{equation}
As a first step, we apply \cite[Lemma 5.6]{FI} to \eqref{def:reduced_composition} and see that
\begin{equation}
G  = \dsum_{\gcd(b_1,b_2)=1} \theta_{b_1}' \theta_{b_2}'' g'(b_1) g''(b_2) \prod_{p \nmid b_1b_2} (1- g'(p)-g''(p)). 
\label{eqn:invert_sieve}
\end{equation}
Define $\tilde{h}', \tilde{h}''$ and $\tilde{g}', \tilde{g}''$ to be multiplicative functions supported on squarefree numbers with
\[
\tilde{h}'(p) =  \frac{g'(p)}{1-g'(p)-g''(p)}, \quad \tilde{h}''(p) =  \frac{g''(p)}{1-g'(p)-g''(p)}, \quad  \tilde{g}'(p) = \frac{g'(p)}{1-g''(p)}, \quad \tilde{g}''(p) = \frac{g''(p)}{1-g'(p)}.
\]
Thus we obtain the usual relations
\begin{equation}
\tilde{h}'(p) =  \frac{\tilde{g}'(p)}{1-\tilde{g}'(p)} \quad \text{and} \quad \tilde{h}''(p) =  \frac{\tilde{g}''(p)}{1-\tilde{g}''(p)}.
\label{eqn:sieve_relations}
\end{equation}
Note $\tilde{h}'(p), \tilde{h}''(p) \geq 0$ and $0 \leq \tilde{g}'(p), \tilde{g}''(p) < 1$ by \eqref{eqn:size_g}. Inserting these definitions into \eqref{eqn:invert_sieve}, we observe that
\[
G = \Big(\prod_p (1-g'(p)-g''(p))\Big) \dsum_{\gcd(b_1,b_2)=1} \theta_{b_1}' \theta_{b_2}'' \tilde{h}'(b_1) \tilde{h}'(b_2). 
\]
If $\gcd(b_1,b_2) \neq 1$ then the expression $\theta_{b_1}' \theta_{b_2}'' \tilde{h}'(b_1) \tilde{h}''(b_2)$ is non-positive by \eqref{eqn:lowerbound_sieve}, so we may introduce all of these terms at the cost of a lower bound for $G$. Thus
\begin{equation}
G \geq \Big(\prod_p (1-g'(p)-g''(p))\Big) \Big( \sum_{b_1} \theta'_{b_1} \tilde{h}'(b_1) \Big) \Big( \sum_{b_2} \theta''_{b_2} \tilde{h}''(b_2) \Big). 
\label{eqn:G_lowerbound}
\end{equation}
The two sums in \eqref{eqn:G_lowerbound} are prepared for standard beta-sieve analysis. 

\begin{lem} \label{lem:FundamentalLemma}
	If $\Lambda'$ is a lower bound beta-sieve with $\beta = 9\kappa + 1$ and $s \geq \beta$ then
	\[
	\sum_b \theta_b' \tilde{h}'(b) \geq 1 - e^{9\kappa - s} K^{10}.
	\]
	If $\Lambda''$ is an upper bound beta-sieve with $\beta = 9\kappa+1$ and $s \geq \beta$ then
	\[
	\sum_b \theta_b'' \tilde{h}''(b) \leq 1 + e^{9\kappa -s} K^{10}.
	\]
\end{lem}
\begin{proof}
	This statement is essentially the Fundamental Lemma \cite[Lemma 6.8]{FI}. To make the comparison clear with \cite[Sections 6.3--6.5]{FI}, one begins with \cite[Equation 6.40]{FI} with their $D, h, g$ replaced by our $R, \tilde{h}', \tilde{g}'$ (or $R, \tilde{h}'', \tilde{g}''$, respectively). Per the definition of $V(z)$ on \cite[p. 56]{FI}, it follows that
	\[
	V(z) = \prod_{p < z} (1-\tilde{g}'(p)) .
	\]
	Thus the assumption \cite[Equation 5.38]{FI} corresponds to our \eqref{eqn:sieve_dimension}. Next, one defines $V_n$ just as in the equation at the top of \cite[p. 63]{FI}; in doing so, we obtain \cite[Equations 6.43 and 6.44]{FI}. Finally, using the same truncation parameters, the analysis of \cite[Section 6.5]{FI} leading up to \cite[Lemma 6.8]{FI} yields our result.
\end{proof}

Now, we apply \cref{lem:FundamentalLemma} to the sum over $b_1$ (the lower bound sieve $\Lambda'$) in \eqref{eqn:G_lowerbound}. Note that the assumption $s > 9\kappa + 1 + 10\log K$ implies that this sum over $b_1$ is positive. By the positivity of $\tilde{h}$ and \eqref{eqn:lowerbound_sieve}, we may trivially estimate the sum over $b_2$ in \eqref{eqn:G_lowerbound} by
\[
\sum_{b_2}  \tilde{h}''(b_2) \,\theta''_{b_2} \geq  \tilde{h}''(1) \, \theta''_1 = 1. 
\]
This proves the lower bound in \cref{thm:CompositionBetaSieves}. For the upper bound, we follow the same arguments and apply \cref{lem:FundamentalLemma} twice (once to each sieve) in these final steps. \hfill \qed

\section{Restricted primes represented by binary quadratic forms	}
\label{sec:BQF}

We recall the setup in \cref{subsec:application}.  Let
\[
f(u,v) = au^2 + buv + cv^2 \in\Z[u,v]
\]
be a positive definite binary quadratic form of discriminant $D = b^2 - 4ac<0$, not necessarily fundamental. The group $\SL_2(\Z)$ naturally acts on such forms by $(T \cdot f)(\mathbf{x}) = f(T \mathbf{x})$ for $T \in \SL_2(\Z)$.  The class number $h(D)$ is the number of such forms up to $\SL_2$-equivalence.  We assume that $f$ is primitive (that is, $\gcd(a,b,c) = 1$), and we define
\[
\mathrm{stab}(f) = \{ T \in \SL_2(\Z) : T \cdot f = f \}.
\]
Note $|\stab(f)| = 2$ unless $D = -3$ or $-4$ in which case it equals $6$ and $4$ respectively. 

\subsection{Proof of \cref{thm:application}} Let $1 \leq R \leq x^{1/10}$ be a parameter yet to be specified. Let $\Lambda' = (\lambda_d')_d$ and $\Lambda''=(\lambda_d'')_d$ be sieve weights supported on squarefree integers $d \mid P$  satisfying
\begin{equation}
\lambda_1' = \lambda_1'' = 1, \qquad |\lambda_d'| \leq 1, \qquad |\lambda_d''| \leq 1 \quad \text{for $d \geq 1$},  \qquad \lambda_d' = \lambda_d'' = 0 \quad \text{for $d \geq R$}. 
\label{eqn:sieve_constraint}
\end{equation}
We approximate the condition $(uv,P)=1$ in \eqref{eqn:BQF_primes} by considering the sieved sum
\begin{equation}
	\label{eqn:sum_1}
	S(x) = S(x; \Lambda', \Lambda'') := \frac{1}{|\stab(f)|} \dsum_{\substack{u,v\in\Z \\ f(u,v)\leq x}}\mathbf{1}_{\mathbb{P}}(f(u,v))\Big(\sum_{d_1\mid u}\lambda_{d_1}'\Big)\Big(\sum_{d_2\mid v}\lambda_{d_2}''\Big). 
\end{equation}
By swapping the order of summation, 
\begin{equation}
	\label{eqn:sum_2}
	S(x) = \dsum_{\substack{d_1, \, d_2 \\ \gcd(d_1,d_2)=1}}\lambda_{d_1}'\lambda_{d_2}''\,  A_{d_1,d_2}(x), 
\end{equation}
where
\begin{equation}
\label{def:bqf_congruence_sum}
A_{d_1,d_2}(x) = \frac{1}{|\stab(f)|} \dsum_{\substack{f(u,v)\leq x \\ d_1\mid u, \, \,d_2\mid v}}\mathbf{1}_{\mathbb{P}}(f(u,v)).
\end{equation}
Before computing the congruence sums $A_{d_1,d_2}(x)$, we introduce the local densities $g'$ and $g''$.  These are multiplicative functions defined by
\begin{equation}
\begin{aligned}
	g'(p) & = \begin{cases}
 		\big( p - (\frac{D}{p}) \big)^{-1} & \text{if $p \mid P$ and $p \nmid c$,} \\ 
 		0 & \text{otherwise,}
 \end{cases} \\
 g''(p) & = \begin{cases}
 		\big( p - (\frac{D}{p}) \big)^{-1} & \text{if $p \mid P$ and $p \nmid a$,} \\ 
 		0 & \text{otherwise.}
 \end{cases}
 \end{aligned}
 \label{def:bqf_local_density}
\end{equation}
Here $(\frac{D}{p})$ is the usual Legendre symbol for $p \neq 2$ and 
\begin{equation}
	\label{def:Legendre_2}
	\Big(\frac{D}{2}\Big)=\begin{cases}
		0&\mbox{if $2\mid D$,}\\
		1&\mbox{if $D\equiv 1\pmod{8}$,}\\
		-1&\mbox{if $D\equiv 5\pmod{8}$.}
	\end{cases}
\end{equation}
Our main result on the Chebotarev density theorem (\cref{thm:main_theorem})  yields the following key lemma whose proof is postponed to \cref{subsec:bqf_lem}.
\begin{lem}
	\label{lem:bqfs_congruence_sum}
	Let $\gamma > 0$ and $\vartheta>0$ be a sufficiently small absolute constants, and let $d_1, d_2$ be relatively prime integers dividing $P$. If $|d_1d_2 D| \leq x^{\gamma}$ then
\begin{equation}
\begin{aligned}
	A_{d_1,d_2}(x)  = g'(d_1) g''(d_2) \frac{\Li(x) - \Li(x^{\beta_1})}{h(D)} \{1 + O(\epsilon_{d_1d_2}(x))\} + O( \sqrt{x} \log x),
\end{aligned}
\label{eqn:bqf_congruence_sum_3}
\end{equation}
where $\beta_1$ is a simple real zero of the Dedekind zeta function $\zeta_{\Q(\sqrt{D})}(s)$ (if it exists) and
\begin{equation}
	\label{eqn:remainder_error}
	\epsilon_d(x) = \epsilon_d(x;D)=\exp\Big[-\vartheta\frac{\log x}{\log|dD|}\Big]+\exp\Big[-(\vartheta \log x)^{1/2}\Big] \quad \text{for $d \geq 1$.}
\end{equation}  
\end{lem}
\begin{remark}
	For the remainder of the proof of \cref{thm:application}, the constant $\vartheta$ may	 be allowed to vary from line-to-line. This will occurs finitely many times, so this is no cause for concern. 
\end{remark}

\begin{remark}
	\label{remark:bqf_cdt}
For the sieve to succeed, one crucially requires an asymptotic equality for $A_{d_1,d_2}(x)$ as in \eqref{eqn:bqf_congruence_sum_3} with small remainder terms. Proceeding via the Chebotarev density theorem, one might use a stronger version of \eqref{eqn:CDT} in \cite{VKM} to obtain the asymptotic
	\begin{equation}
	\label{eqn:CDT_BQF_Murty}
	A_{d_1,d_2}(x)  = \frac{g'(d_1) g''(d_2)}{h(D)}(\Li(x) +O(xe^{-\Cr{CDT}\sqrt{\log x}})),\quad \text{for } \log x\gg (\log|d_1 d_2 D|)^2+\frac{1}{1-\beta_1}.
	\end{equation}
	Currently, $(1-\beta_1)^{-1}\ll |D|^{1/2}\log|D|$ is the best unconditional effective bound for $\beta_1$.  Thus $x$ must be quite large with respect to $|D|$, $d_1$, and $d_2$; this adversely impacts the permissible ranges of $|D|$ and $z$ in \cref{thm:application}. To improve the range of $x$, one might instead appeal to variants of \eqref{eqn:LMO_Linnik} found in \cite{TZ1,TZ2,Weiss} but this only yields lower and upper bounds for $A_{d_1,d_2}(x)$, rendering the sieve powerless. Fortunately, \cref{thm:main_theorem} addresses all of these obstacles simultaneously. Regardless of whether $\beta_1$ exists, it maintains an asymptotic with an improved range of $x$ that is polynomial in $|D|, d_1,$ and $d_2$ while keeping satisfactory control on the error terms. This allows us to  strengthen the uniformity  of both $z$ and $|D|$ in \cref{thm:application} beyond what earlier versions of the Chebotarev density theorem permit. 
\end{remark}

Now,  set the level of distribution to be
\begin{equation}
R := z^{ \frac{1}{\sqrt{\eta}} \log \log z}.
\label{eqn:bqf_level_distribution}
\end{equation}
Since $z \leq x^{\eta/\log\log x}$ and $|D| \leq x^{\eta/\log\log z}$ by assumption, we have that $R \leq x^{1/10}$ and also $|d_1d_2D|  \leq x^{4\sqrt{\eta}}$ for any integers $d_1, d_2 < R$. Thus,  by \cref{lem:bqfs_congruence_sum} and \eqref{eqn:sieve_constraint}, it follows that
\begin{equation}
S(x) = \big( \mathcal{G}  + O(\mathcal{R})	\big) \frac{\Li(x)-\Li(x^{\beta_1})}{h(D)} + O(x^{3/4}),
\label{eqn:bqf_sifted}
\end{equation}
where
\begin{equation*}
\begin{aligned}
\mathcal{G} = \dsum_{\substack{d_1, d_2 \\ \gcd(d_1,d_2)=1}} \lambda_{d_1}'\lambda_{d_2}'' g'(d_1) g''(d_2), \qquad 
	\mathcal{R}
		& = \sum_{\substack{d < R^2 \\ d \mid P}} \frac{\tau(d)}{\varphi(d)} \epsilon_d(x).
\end{aligned}	
\end{equation*}
Here $\tau$ is the divisor function and $\varphi$ is Euler phi function.  We proceed to calculate the main term $\mathcal{G}$ and remainder terms $\mathcal{R}$.

\subsubsection{Main term $\mathcal{G}$} 
\label{subsec:bqf_main_term}
	For the main term $\mathcal{G}$, suppose we have chosen a lower bound sieve for the sum in \eqref{eqn:BQF_primes}; namely, suppose $\Lambda'$ is a lower bound beta sieve and $\Lambda''$ is an upper bound beta sieve, each with level of distribution $R$. Our aim is to apply the Fundamental Lemma in the form of \cref{thm:CompositionBetaSieves}. One can see that $g'$ and $g''$ are each  satisfy \eqref{eqn:sieve_dimension} with $\kappa = 1$ and $K$ absolutely bounded. Moreover, our choice of sieve has a sufficiently large sifting variable $s = \frac{\log R }{\log z} \gg \eta^{-1}$ because  $\eta > 0$ is sufficiently small.
	
	We claim that we may assume  
	\[
	g'(p) + g''(p) < 1 \qquad \text{for all primes $p$}
	\]
	and hence $g'$ and $g''$ also satisfy \eqref{eqn:size_g}. From \eqref{def:bqf_local_density}, the only concern occurs when $p = 2$ and $2 \mid P$. We prove the claim by checking cases and verifying that $g'(2) + g''(2) \geq 1$ only if \cref{thm:application} is trivially true. 
	\begin{itemize}
		\item Suppose $D \equiv 5 \pmod{8}$. By \eqref{def:bqf_local_density}, we have $g'(2) + g''(2) \leq \frac{1}{3} + \frac{1}{3} < 1$. 
		\item Suppose $D \equiv 1 \pmod{8}$ so $b \equiv 1 \pmod{2}$ and $ac \equiv 0 \pmod{2}$. If $a+b+c \equiv 0 \pmod{2}$ then  the sum in \eqref{eqn:BQF_primes} is necessarily empty because $\mathbf{1}_{\mathbb{P}}$ only detects odd primes. In this case, $a$ and $c$ have opposite parity so $g'(2) + g''(2) = 1$. Hence, $\delta_{f}(P) = 0$ by \eqref{eqn:bqf_euler_product} and \cref{thm:application} is therefore trivially true.  Otherwise, if $a + b + c \equiv 1 \pmod{2}$ then $a$ and $c$ have the same parity. As $ac \equiv 0 \pmod{2}$, it must be that $a \equiv c \equiv 0 \pmod{2}$ implying $g'(2) + g''(2) = 0 < 1$ by definition \eqref{def:bqf_local_density}. 
		\item Suppose $2 \mid D$ so $b \equiv 0 \pmod{2}$. If one of $a$ or $c$ is even then $g'(2) + g''(2) \leq \frac{1}{2} < 1$. Otherwise, if both $a$ and $c$ are odd then $g'(2) + g''(2) = 1$ and $a+b+c \equiv 0 \pmod{2}$. This implies  $\delta_{f}(P) = 0$ and also the sum in \eqref{eqn:BQF_primes} is necessarily empty so \cref{thm:application} is trivially true. 
	\end{itemize}
This proves the claim. Therefore, by \cref{thm:CompositionBetaSieves} and \eqref{eqn:bqf_level_distribution}, it follows that
\begin{equation}
\mathcal{G} \geq \delta_{f}(P)\{ 1 + O_A( (\log z)^{-A}) \}
\label{eqn:bqf_main_term}
\end{equation}
since $\eta = \eta(A)$ is sufficiently small. If $\Lambda'$ and $\Lambda''$ are both upper bound beta sieves with level of distribution $x^{1/10}$ then one similarly obtains the reverse inequality. 

\subsubsection{Remainder terms $\mathcal{R}$}
  We estimate $\mathcal{R}$ dyadically. By the Cauchy-Schwarz inequality and standard estimates for $\tau$ and $\varphi$, we see for $0 \leq N \leq \lceil  \frac{2\log R}{\log z} \rceil$ that  
\begin{equation*}
\begin{aligned}
		\sum_{\substack{z^N \leq d < z^{N+1} \\ d \mid P}} \frac{\tau(d)}{\varphi(d)} \epsilon_d(x) 
		& \ll \epsilon_{z^{N+1}}(x) \Big( \sum_{\substack{z^N \leq d < z^{N+1} \\ p \mid d \implies p \leq z}} \frac{1}{d}\Big)^{1/2} \Big( \sum_{\substack{z^N \leq d < z^{N+1}}} \frac{\tau(d)^2 d}{\varphi(d)^2}\Big)^{1/2} \\
		& \ll \epsilon_{z^{N+1}}(x) ( (N+1) \log z) ^{3/2}\Big( \sum_{\substack{z^N \leq d < z^{N+1} \\ p \mid d \implies p \leq z}} \frac{1}{d}\Big)^{1/2}.\\ 
\end{aligned}
\end{equation*}
By \eqref{eqn:bqf_level_distribution}, one has that $R^{\eta'/\log\log R} \leq z \leq R$ where $\eta' > 0$ is sufficiently small depending only on $\eta$. In other words, $\frac{\log R}{\log z} \ll \log\log z$. Thus, we may apply 
Hildebrand's estimate for $z$-smooth numbers \cite[Theorem 1]{Hildebrand-1986} via partial summation to conclude from \eqref{eqn:remainder_error} that the above is
\[
\ll (e^{-\vartheta \frac{\log x}{(N+1)\log z}} + e^{-\vartheta \frac{\log x}{\log |D|}} + e^{-\vartheta \sqrt{\log x}} )\rho(N) (N+1)^2 \log^2 z,
\]
where $\rho$ is the Dickman-de Bruijn function. Recall we allow the constant $\vartheta > 0$ to change from line-to-line and be replaced by a smaller value if necessary. Summing this estimate over $0 \leq N \leq \lceil \frac{2 \log R}{\log z} \rceil$ and using the crude estimate $\rho(N) \ll N^{-N}$ for $N \geq 1$, we deduce that
\begin{equation*}
\begin{aligned}
\mathcal{R} 
	& \ll ( \max_{N \geq 1} e^{- \frac{c \log x}{N\log z}} N^{-N + 2} ) \log^2 z  +  (e^{-\vartheta \frac{\log x}{\log z}} + e^{-\vartheta \frac{\log x}{\log |D|}} + e^{- \vartheta \sqrt{\log x}})\log^2 z \\
	& \ll ( e^{- \vartheta \sqrt{\frac{\log x \log\log x}{\log z}} } + e^{-\vartheta \frac{\log x}{\log z}}  + e^{- \vartheta \frac{\log x}{\log |D|}} + e^{- \vartheta \sqrt{\log x}})\log^2 z.
\end{aligned}
\end{equation*}
Since $|D| \leq x^{\eta/\log\log z}$ and $z \leq x^{\eta/\log\log x}$ with $\eta = \eta(A) > 0$ sufficiently small, we have that
\begin{equation}
\mathcal{R} \ll_A (\log z)^{-A}. 
\label{eqn:bqf_remainder_estimate}
\end{equation}
 
\subsubsection{Concluding the proof}
Inserting \eqref{eqn:bqf_main_term} and \eqref{eqn:bqf_remainder_estimate} into \eqref{eqn:bqf_sifted} along with the fact that $\delta_{f}(P) \gg (\log z)^{-2}$ from Mertens' estimate, we conclude that

\begin{equation*}
\begin{aligned}
& \dsum_{\substack{u,v \in \Z \\ au^2 + buv + cv^2 \leq x \\ (uv,P)=1}} \frac{\mathbf{1}_{\mathbb{P}}(au^2 + buv + cv^2)}{|\stab(f)|}  \geq \delta_{f}(P) \frac{\Li(x) - \Li(x^{\beta_1})}{h(D)} \{1 + O_A( (\log z)^{-A})\} + O(x^{3/4}). 
\end{aligned}
\end{equation*}
By using an upper bound sieve instead (as mentioned at the end of \cref{subsec:bqf_main_term}), one also obtains the reverse inequality. Thus, it remains to show the secondary error term $O(x^{3/4})$ may be absorbed into the primary error term. If $\delta_{f}(P) = 0$ then the arguments in \cref{subsec:bqf_main_term} imply \cref{thm:application} trivially true so we may assume $\delta_{f}(P) > 0$. By the effective lower bound that $1-\beta_1 \gg_{\epsilon} |D|^{-1/2-\epsilon}$, the fact that $h(D) \ll_{\epsilon} |D|^{1/2+\epsilon}$, and the assumption that $|D| \leq x^{\eta/\log\log z}$, we see
\[
\frac{\Li(x) - \Li(x^{\beta_1})}{h(D)} \gg x^{4/5}.
\]
As $\delta_f(P) \gg (\log z)^{-2}$, this implies the claim and hence proves \cref{thm:application}.  \hfill \qed
\subsection{Proof of \cref{lem:bqfs_congruence_sum}} \label{subsec:bqf_lem} The pair $(d_1,d_2)$  induces another form $f_{d_1,d_2}$   given by 
\[
f_{d_1,d_2}(s,t) := f(d_1s,d_2t).
\]
Note its discriminant is $D(d_1d_2)^2$. With this definition, it follows that
\begin{equation}
A_{d_1,d_2}(x) = \frac{1}{|\stab(f)|} \sum_{p\leq x}\#\{ (s,t) \in \Z^2 : p = f_{d_1,d_2}(s,t)\}. 
\end{equation}
Observe 
\[
A_{d_1,d_2}(x) \ll 1 \qquad \text{if $(d_1,c) \neq 1$ or $(d_2,a) \neq 1$} 
\]
since, in this case, $f_{d_1,d_2}$ is not primitive and hence represents an absolutely bounded number of primes. This trivially establishes \cref{lem:bqfs_congruence_sum} in this case. To evaluate $A_{d_1,d_2}(x)$ for all other $d_1$ and $d_2$, we use class field theory. 

\begin{lem}
\label{lem:bqfs_cft}
Let $\cO_K$ be the ring of integers of $K = \Q(\sqrt{D})$. For $d \geq 1$, let $\cO_d$   be the order of discriminant $-Dd^2$ in $K$ and let $L_d$ be the ring class field of $\mathcal{O}_d$. If $F$ is a primitive binary quadratic form of discriminant $-Dd^2$ then 
	\[
	|\cO^{\times}_d| = |\stab(F)|. 
	\]
	 Moreover, if $C_F$ is the conjugacy class corresponding to $F$ in the Galois group of $L_d/K$ then
	\[
	\#\{ (s,t) \in \Z^2 : p = F(s,t)\} =|\cO_d^{\times}|\cdot\#\{\kp \subseteq \cO_K \colon \N\kp=p,~ [\tfrac{L_d/K}{\kp}] = C_F\} \quad \text{for $p \nmid Dd$}. 
	\]
	Here $[\tfrac{L_d/K}{\kp}]$ is the Artin symbol of $\kp$ and $\N = \N_{K/\Q}$ is the absolute norm of $K/\Q$.  
\end{lem}
\begin{proof} These are straightforward consequences of the theory for positive definite binary quadratic forms, so we only sketch the details. Standard references include for example \cite{Cassels-2008,Cox}.  First, one can verify that $\cO_d^{\times} = \{ \pm 1\}$ unless $\cO_d$ is the ring of integers for $\Q(i)$ or $\Q(\sqrt{-3})$. Similarly, the $SL_2$-automorphism group of $F$ is $\{ \pm \big(\begin{smallmatrix} 1 & 0 \\ 0 & 1 \end{smallmatrix}\big) \}$ unless $F$ is properly equivalent to either $x^2+y^2$ or $x^2 + xy + y^2$. These are respectively the unique  reduced forms of discriminant $-4$ or $-3$. These remaining two cases can be checked by direct calculation. 

The second claim follows from the first claim and the one-to-one correspondence between inequivalent representations of a prime $p$ by $F$ and degree 1 prime ideals $\kp \subseteq \cO_K$ in the class $C_F$. For more details, see \cite[Theorem 7.7]{Cox}.
 \end{proof}

Now, assuming $(d_1,c) = (d_2,a) = 1$, we return to computing $A_{d_1,d_2}(x)$. It follows $f_{d_1,d_2}$ is primitive so by \cref{lem:bqfs_cft} with $F = f_{d_1,d_2}$ and $d= d_1d_2$,we deduce that
\begin{equation}
	\label{eqn:bqf_congruence_sum}
	A_{d_1,d_2}(x) =    \frac{1}{|\stab(f)|} \sumD_{\substack{\N\kp\leq x \\ \deg(\kp)=1}} |\mathcal{O}_{d_1d_2}^{\times}| + O\Big(\sum_{p \mid Dd_1d_2} 1 \Big) ,
\end{equation}
where   $\sum^{\dagger}$ runs over prime ideals $\kp$ in $\mathcal{O}_K$ unramified in $L_{d_1d_2}$ satisfying $[\frac{L_{d_1d_2}/K}{\kp}]= C_{f_{d_1,d_2}}$. Note, for the primes $p \mid Dd_1d_2$ in \eqref{eqn:bqf_congruence_sum}, we have used that each prime $p$ is represented by $f$ with absolutely bounded multiplicity. We may add the remaining degree 2 prime ideals $\kp$ to the $\dagger$-marked sum with error at most $O(|\cO_{d_1d_2}^{\times}|\sqrt{x}\log x) = O(\sqrt{x} \log x)$. Further, we have 
\[
\sum_{p \mid Dd_1d_2} 1 \ll \log|Dd_1d_2| \ll \log x
\]
since $|d_1d_2D| \leq x^{\gamma}$. Collecting these observations, it follows that
\begin{equation}
		\label{eqn:bqf_congruence_sum_2}
	A_{d_1,d_2}(x) = \frac{|\mathcal{O}_{d_1d_2}^{\times}|}{|\stab(f)|}  \sumD_{\substack{\N\kp\leq x}}1  + O(\sqrt{x} \log x).  
\end{equation}

We invoke \cref{thm:main_theorem} to compute the sum in \eqref{eqn:bqf_congruence_sum_2}, thus
\begin{equation}
\sumD_{\substack{\N\kp\leq x} } 1=\frac{\mathrm{Li}(x)- \theta_1 \mathrm{Li}(x^{\beta_1})}{h(D(d_1 d_2)^2)}\{ 1+O(\epsilon_{d_1d_2}(x)) \} \qquad \text{for $|d_1d_2D| \leq x^{\gamma}$,}
\label{eqn:bqf_apply_thm}
\end{equation}
 where $\epsilon_{d_1d_2}(x)$ is defined by \eqref{eqn:remainder_error} and $\gamma > 0$ is fixed and sufficiently small. We make two simplifications for \eqref{eqn:bqf_apply_thm}. First, we claim that $\theta_1 = 1$ if the exceptional zero $\beta_1$ exists. By a theorem of Heilbronn \cite{Heilbronn-1972} generalized by Stark \cite[Theorem 3]{Stark-1974},  since $\beta_1$ is a real simple zero of $\zeta_{L_{d_1d_2}}(s)$ and $L_{d_1d_2}$ is Galois over $\Q$ with $K$ being its only quadratic subfield, it follows that $\zeta_K(\beta_1) = 0$. Hence, the exceptional Hecke character $\chi_1$ of $K$ from \cref{thm:main_theorem} is trivial implying $\theta_1 = 1$. Second, we have for $d \geq 1$ that
\begin{equation}
h(Dd^2) = \frac{h(D)}{[\cO^{\times} : \cO_d^{\times}]} d \prod_{p \mid d} \Big(1 - \Big(\frac{D}{p}\Big) \frac{1}{p} \Big).
\label{eqn:classnumber}
\end{equation}
For a proof, see for example \cite[Theorem 7.4 and Corollary 7.28]{Cox}. 

Finally, with these observations, \cref{lem:bqfs_congruence_sum} follows by inserting \eqref{eqn:bqf_apply_thm} and \eqref{eqn:classnumber} into \eqref{eqn:bqf_congruence_sum_2} and noting that $[\cO_1^{\times} : \cO_d^{\times}] \cdot |\cO_d^{\times}| = |\cO_1^{\times}| = |\stab(f)|$ from \cref{lem:bqfs_cft}. 
\hfill \qed
%
%

 \appendix
\section{Error term with an exceptional zero}
\label{sec:appendix}

\cref{thm:LFZDE} states that if $T\geq 1$, then
\begin{equation}
\label{eqn:LFZDE_with_siegel}
\sum_{\chi}N(\sigma,T,\chi)\ll B_1(QT^{n_K})^{\Cr{ZDE}(1-\sigma)},\quad B_1=\min\{1,(1-\beta_1)\log(QT^{n_K})\}.
\end{equation}
This clearly implies that regardless of whether $\beta_1$ exists, we have
\begin{equation}
\label{eqn:LFZDE_no_siegel}
\sum_{\chi}N(\sigma,T,\chi)\ll (QT^{n_K})^{\Cr{ZDE}(1-\sigma)}.
\end{equation}
If $\beta_1$ exists, \cref{thm:LFZDE} produces the following strong zero-free region:
 
 \begin{thm}[Zero repulsion] \label{thm:DH}
	Suppose the exceptional zero $\beta_1$ of \cref{thm:ZFR} exists. There exists $\Cl[abcon]{DH} > 0$ such that if $\Delta$ is given in \cref{thm:ZFR}, then
	\[
	\Delta(t) \geq  \min\Big\{ \frac{1}{2}, \frac{ \Cr{DH} \log\big(\big[ (1-\beta_1) \log(Qt^{n_K}) \big]^{-1}\big) }{\log(Qt^{n_K})} \Big\}.
	\]
\end{thm}

Let $q\geq 1$ be an integer.  In the context of arithmetic progressions, in which case $L=\Q(e^{2\pi i/q})$ and $F=K=\Q$, it is preferable to use \eqref{eqn:LFZDE_no_siegel} and \cref{thm:DH} instead of \eqref{eqn:LFZDE_with_siegel}, as one can typically obtain numerically superior results with the former.  However, in the context of arithmetic progressions, one has the benefit of working with characters of an extension which is abelian over $\Q$, in which case \cref{thm:Stark} gives an adequate upper bound for $\beta_1$ (should it exist).  However, for abelian extensions $L/K$ where the root discriminant of $K$ is rather small, \cref{thm:Stark} gives an upper bound for $\beta_1$ which is not commensurate with the corresponding result for cyclotomic extensions of $\Q$.  In fact, this weak upper bound leads us to actually require a version of the log-free zero density estimate that improves as $\beta_1$ approaches 1 to handle the case when $K$ has a small root discriminant.  This is why we use \eqref{eqn:LFZDE_with_siegel} in our proofs instead of using \eqref{eqn:LFZDE_no_siegel} and \cref{thm:DH} separately.

For comparison with \cref{lem:ErrorOptimization_Classical}, we quantify the effect of \eqref{eqn:LFZDE_no_siegel} and \cref{thm:DH} on the error term in \cref{lem:SumOverZeros} and subsequently \eqref{eqn:WeightPrimeSum_Penultimate} in the proof of \cref{prop:WeightedPrimeSum}.  Since the calculations are tedious, we omit the proof.

\begin{lem}
	\label{lem:ErrorOptimization_SiegelZero}	
	Let $\eta$ be defined by \eqref{def:eta_function}. Suppose the exceptional zero $\beta_1 = 1 - \frac{\lambda_1}{\log Q}$ of \cref{thm:ZFR} exists. There exists absolute constants $\Cl[abcon]{SZ_lambda_size}, \Cl[abcon]{SZ_error}, \Cl[abcon]{SZ_size} > 0$ such that if $\lambda_1 \leq \Cr{SZ_lambda_size}$ and 
	$Q \leq x^{1/\Cr{SZ_size}}$, 
	\begin{align}
		e^{-\eta(x)} & \ll x^{-1/2} + \lambda_1^{10} \Big( e^{-\frac{\Cr{DH} \log x}{2 \log Q}} +  e^{- \Cr{SZ_error} \sqrt{ (\log x)/n_K}} \Big) & \text{if $\lambda_1 \geq Q^{-20/n_K}$}, 
				\label{eqn:ErrorTerm_Siegelzero_2} \\
		e^{-\eta(x)} & \ll x^{-1/2} + e^{-10\sqrt{\log(1/\lambda_1) } } \Big(  e^{-\frac{\Cr{DH} \log x}{2 \log Q}} +  e^{- \Cr{SZ_error} \sqrt{ (\log x)/n_K}}  \Big) & \text{if $\lambda_1 < Q^{-20/n_K}$}. 
			\label{eqn:ErrorTerm_Siegelzero_3}	
	\end{align}

\end{lem}
 
\begin{remark}
\label{rem:SZ}
Recall the definition of $\nu_1$ in \eqref{def:nu}.  From \eqref{eqn:proof_WeightedPrimeSum} and  \eqref{eqn:MainTerm_LB}, one can see it is critical to prove an estimate at least as strong as
\begin{equation}
\nu_1 x e^{-\eta(x)} = o(\lambda_1 x). 
\label{eqn:SZ_MinError}
\end{equation}
Notice that the density estimate in \eqref{eqn:LFZDE_with_siegel} decays linearly with respect to $1-\beta_1$ (that is, $\nu_1 = \lambda_1$), so we easily obtain \eqref{eqn:SZ_MinError}. Suppose we instead use \eqref{eqn:LFZDE_no_siegel}, which is tantamount to the trivial estimate $\nu_1\leq 1$ when $\beta_1$ exists.  From \eqref{eqn:ErrorTerm_Siegelzero_2}, one obtains \eqref{eqn:SZ_MinError} when $\lambda_1 \geq Q^{-20/n_K}$. Otherwise, from \eqref{eqn:ErrorTerm_Siegelzero_3}, if $\lambda_1 < Q^{-20/n_K}$ then we can at best show $x e^{-\eta(x)} = o(e^{-10\sqrt{\log(1/\lambda_1) } } x)$.  The situation $\lambda_1 < Q^{-20/n_K}$ is not uniformly excluded by Stark's bound \eqref{eqn:Stark}. For example, when the root discriminant $D_K^{1/n_K}$ is bounded and the extension $L/K$ is unramified (that is, $\cQ = 1$), then 
\[
Q^{100/n_K} = (D_K \cQ)^{100/n_K} n_K^{100} \ll n_K^{100}
\]
 and Stark's bound \eqref{eqn:Stark} implies $\lambda_1^{-1} \ll n_K^{n_K} \log D_K$ so it may very well be the case that $\lambda_1^{-1} \gg n_K^{100} \gg Q^{100/n_K}$.  This situation with a bounded root discriminant is entirely possible as Minkowski's unconditional estimate $n_K \ll \log D_K$ is tight when varying over all number fields $K$. Infinite class field towers are well known sources of this scenario.  Thus, we cannot see how to unconditionally obtain the desired linear decay demanded by \eqref{eqn:SZ_MinError} with only \eqref{eqn:LFZDE_no_siegel} and \cref{thm:DH}.  
\end{remark}

\bibliographystyle{abbrv}
\bibliography{CDT_unified}
\end{document}